\documentclass[12pt,leqno,a4paper]{article}
\usepackage{tikz}
\usepackage{comment}
\usepackage{amsmath}
\usepackage{amsfonts}
\usepackage{amsthm}
\usepackage{amssymb}
\usepackage[showonlyrefs]{mathtools}
\mathtoolsset{showonlyrefs}
\def\C{\mathbb C}

\def\N{\mathbb N}
\def\Z{\mathbb Z}
\def\R{\mathbb R}
\def\Q{\mathbb Q}
\def\OO{{\mathcal O}}
\def\deg{\operatorname{deg}}

\newtheorem{theorem}{Theorem}[section]
\newtheorem{lemma}[theorem]{Lemma}

\theoremstyle{remark}

\newtheorem*{ack}{Acknowledgment}
\newtheorem*{rem}{Remark}
\newtheorem*{case1}{Case 1}
\newtheorem*{case2}{Case 2}
\newtheorem{remark}[theorem]{Remark}
\numberwithin{equation}{section}
\begin{document}
\title{Zeros, growth and Taylor coefficients of entire solutions of linear $q$-difference equations}
\author{Walter Bergweiler}
\date{}
\maketitle
\centerline{\emph{Dedicated to the memory of Professor Walter K.\ Hayman}}
\begin{abstract}
We consider transcendental entire solutions of linear $q$-difference equations with polynomial
coefficients and determine the asymptotic behavior of their Taylor coefficients.
We use this to show that under 
a suitable hypothesis on the associated Newton-Puiseux diagram their zeros
are asymptotic to finitely many geometric progressions.
We also sharpen previous results on the growth rate of entire solutions.

\medskip 


\medskip 

\end{abstract}

\section{Introduction and main results} \label{intro}
An equation of the form
\begin{equation}
\sum_{j=0}^m a_j(z)f(q^jz)=b(z)\label{101}
\end{equation}
is called a linear $q$-difference equation.
Here we assume that $q\in\C$ with $0<|q|<1$ and that $b$ and the $a_{j}$
are polynomials.

The study of such equations has a long history -- already Adams' survey~\cite{Adams1931} from 1931
has an extensive bibliography. The subject continues to be an active area of research,
see~\cite{Annaby2012,Barnett2007,Cao2019,Chiang2018,Dreyfus2015,Ramis2013} for a (very incomplete)
 sample of more recent work.

The asymptotics of the Taylor coefficients and the zeros of solutions of this equation were
studied in~\cite{Bergweiler2003} under a hypothesis on the associated Newton-Puiseux diagram.
In particular, it was shown that under this hypothesis the zeros are asymptotic to certain
geometric progressions.
We will show that this holds in much more general situations. Moreover, the asymptotics of the
coefficients are determined without additional hypotheses on the Newton-Puiseux diagram.
We also refine previous estimates of the growth of solutions.

We recall the definition of the Newton-Puiseux diagram $P$ associated to~\eqref{101}. 
Let $d(j)$ denote the degree of~$a_j$. Then $P$ is defined as the convex hull of
\begin{equation} \label{1a}
\bigcup_{j=0}^m \left\{(x,y) \in\R^2\colon x \geq j \  \text{and} \  y \leq d(j) \right\}.
\end{equation}
Let $(j_{k}, d({j_{k}}))$ be the vertices of $P$, with $k\in\{0,\dots, K\}$ and
\begin{equation} \label{1c}
0=j_{0}<j_{1}<\ldots<j_{K}\le m.
\end{equation}
If~\eqref{101} has a transcendental entire solution,
then there exists $j\in\{1,\dots,m\}$ such that $d(j)>d(0)$; see \cite[Theorem 1.1]{Bergweiler2002}.
This implies that $K\geq 1$. For $k\in\{1,\dots,K\}$ we define 
\begin{equation} \label{1c1}
\sigma_{k}:=\frac{d({{j}_{k}})-d({{j}_{k-1}})}{j_{k}-j_{k-1}}.
\end{equation}
Then $\sigma_{1}>\sigma_{2}>\dots>\sigma_{K}>0$. 
The $\sigma_{k}$ are the slopes of the segments which form the boundary of $P$.

If $f$ is a transcendental entire solution of (\ref{101}), then
(\cite[Theorem 1.2]{Bergweiler2002}; see also~\cite[Theorem 4.8]{Ramis1992})
there exists $k\in\{1,\dots K\}$ such that the maximum modulus
\begin{equation} \label{103}
M(r,f):= \max_{|z|=r} |f(z)|
\end{equation}
satisfies
\begin{equation} \label{104}
\log M(r,f)\sim \frac{\sigma_k}{-2\log |q|}\left(\log r\right)^2
\end{equation}
as $r\to\infty$.
The condition posed in~\cite{Bergweiler2003} was that
the segment of the boundary of $P$ whose slope is $\sigma_k$ contains no
point $(j,d(j))$ except for its endpoints $(j_{k-1},d(j_{k-1}))$ and $(j_{k},d(j_{k}))$.

Consider for example the case $m=4$, $\deg a_0=0$,
$\deg a_1=1$, $\deg a_2=2$, $\deg a_3=1$ and $\deg a_4=3$.
The corresponding Newton-Puiseux diagram 
is shown in Figure~\ref{np-diagram2}, with the points $(j,d(j))$ marked.
Here we have $K=2$, $j_1=2$, $j_2=4$, $\sigma_1=1$ and $\sigma_2=1/2$.
The hypothesis on the segment of $\partial P$ corresponding to $\sigma_k$ that
was posed in~\cite{Bergweiler2003} is satisfied for $k=2$, but not for $k=1$.
\begin{figure}[!htb]
\centering
\begin{tikzpicture}[scale=1.2,>=latex](-0.1,-0.1)(5.1,3.1)
\filldraw[gray!10] (0,0) -- (2,2) -- (4,3) -- (5,3) -- (5,-0.7) -- (0,-0.7);
\draw[thick,-] (0,0) -- (2,2) -- (4,3);
\draw[->] (-0.3,0) -- (5.0,0);
\draw[->] (0,-0.7) -- (0,3.3);
\foreach \x in {1,...,4}
   {
    \draw (\x,-2pt) -- (\x,2pt);
    \node[below] at (\x,-2pt) {$\x$};
   }
\foreach \y in {1,...,3}
   {
    \draw (-2pt,\y) -- (2pt,\y);
    \node[left] at (-2pt,\y) {$\y$};
   }
\filldraw[black] (0,0) circle (0.05);
\filldraw[black] (1,1) circle (0.05);
\filldraw[black] (2,2) circle (0.05);
\filldraw[black] (3,1) circle (0.05);
\filldraw[black] (4,3) circle (0.05);
\end{tikzpicture}
\caption{A Newton-Puiseux diagram.}
\label{np-diagram2}
\end{figure}
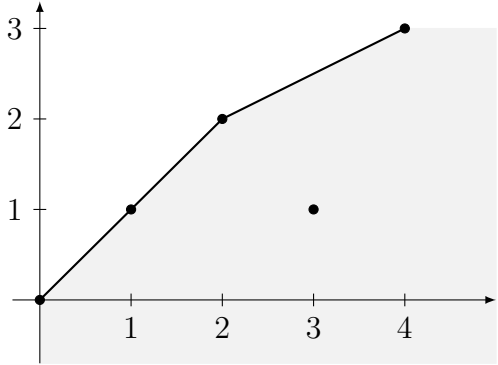

Let now
\begin{equation} \label{105}
f(z)=\sum_{n=0}^\infty \alpha_n z^n
\end{equation}
be a transcendental entire solution of~\eqref{101}. Then
there exists $k\in\{1,\dots,K\}$ such that the growth of $f$ is 
given by~\eqref{104}. With this value of $k$ we put
\begin{equation} \label{MN}
M:=d(j_k)-d(j_{k-1}) 
\quad\text{and}\quad 
N:= j_k-j_{k-1}
\quad\text{so that}\quad 
\sigma_k=\frac{M}{N}.
\end{equation}
We also put
\begin{equation} \label{defrho}
\rho:= q^{1/(2M)},
\end{equation}
for some fixed branch of the root.

It was shown in~\cite[Theorem~2]{Bergweiler2003} that
if the above condition on the segment of $\partial P$ corresponding to $\sigma_k$ is satisfied,
then there exist $\xi\in\C\setminus\{0\}$ and
$(\eta_0,\eta_1,\dots,\eta_{M-1})\in\C^M\setminus\{(0,0,\dots,0)\}$
such that for $r\in\{0,1,\dots,M-1\}$ we have
\begin{equation} \label{106}
\alpha_n= \left(\eta_r+\OO\!\left(|\rho|^{2n}\right)\right) \rho^{N n^2}\xi^n
\end{equation}
as $n\to \infty$ while satisfying $n\equiv r\pmod M$,
and the set of zeros of $f$ can be written in the form
\begin{equation} \label{106a}
\left\{z_{n,\nu}\colon\nu\in\{1,2,\dots,M\},n\in\N\right\}
\end{equation}
such that for each $\nu$ there exists $A_\nu\in\C\setminus\{0\}$
with $z_{n,\nu}\sim A_\nu q^{-N n}$ as $n\to\infty$.
More precisely, 
\begin{equation}\label{107b}
z_{n,\nu}= A_\nu q^{-N n}\left(1+\OO\!\left(|q|^{n/M}\right)\right).
\end{equation}

Let now $I_k$ be the set of all $j\in\{0,\dots,m\}$ for which $(j,d(j))$ lies on the 
line through $(j_{k-1},d(j_{k-1}))$ and $(j_{k},d(j_{k}))$.
It follows from the definition of the Newton-Puiseux diagram that 
\begin{equation} \label{301a}
\left\{j_{k-1},j_k\right\}\subset I_k \subset \left\{j_{k-1},j_{k-1}+1,\dots,j_k-1,j_k\right\} .
\end{equation}
The hypothesis posed in \cite{Bergweiler2003} then takes the form $I_k=\{j_{k-1},j_k\}$, but we shall
drop this condition now.  Write 
\begin{equation}\label{a_j}
a_j(z)=\sum_{i=0}^{d(j)}a_{j,i} z^i 
\end{equation}
and put
\begin{equation} \label{defc}
c_{j,i}:=a_{j,i}\rho^{N i^2}q^{-ji} =a_{j,i}\rho^{N i^2 -2M ji} .
\end{equation}
Then $a_{j,d(j)}\neq 0$ and hence $c_{j,d(j)}\neq 0$ since $d(j)=\deg a_j$. 
The polynomial
\begin{equation}\label{p_k}
P_k(z):=\sum_{j\in I_k} c_{j,d(j)} z^{d(j)-d(j_{k-1})}
\end{equation}
is called (cf.~\cite[p.~511]{Adams1928})
the \emph{characteristic polynomial} associated to $\sigma_k$.
Since, by the definition of the Newton-Puiseux diagram and the set $I_k$, we have 
\begin{equation}\label{bound_dj}
\begin{aligned}
d(j)-d(j_{k-1})
&= (j-j_{k-1}) \sigma_k= \frac{j-j_{k-1}}{j_k-j_{k-1}} \left(d(j_k)-d(j_{k-1})\right)
\\ &
\leq d(j_k)-d(j_{k-1})=M
\end{aligned}
\end{equation}
for all $j\in I_k$, with equality only for $j=j_k$, it follows that $\deg P_k=M$.

We also note that $P(0)=c_{j_{k-1},d(j_{k-1})}\neq 0$.

\begin{theorem} \label{coeff}
Let $f$ be a transcendental entire solution of \eqref{101} with growth given by~\eqref{104} 
and Taylor series~\eqref{105}.
Let $M$, $N$ and $\rho$ be as in~\eqref{MN} and~\eqref{defrho}, let
$P_k$ be the characteristic polynomial associated to $\sigma_k$
as defined in~\eqref{p_k}  and let $\lambda_1,\dots,\lambda_l$ be the 
roots of $P_k$, with multiplicities $m_1,\dots,m_l$ so that
\begin{equation}\label{sum_mj}
\sum_{j=1}^l m_j=M.
\end{equation}

Then there exist $\mu\in \{|\lambda_j|\colon 1\leq j\leq l\}$  
and for each $j$ with $|\lambda_j|=\mu$ a polynomial $Q_j$ with $\deg Q_j\leq m_j-1$,
not all $Q_j$ vanishing identically, such that
\begin{equation} \label{asymp-coeff}
\alpha_n\rho^{-N n^2}
=\sum_{\{j\colon |\lambda_j|=\mu\}} Q_j(n)\lambda_j^n +\OO\!\left(\delta^n\right)
\end{equation}
for some $\delta\in (0,\mu)$ as $n\to\infty$. In fact, this holds for any $\delta$ satisfying
\begin{equation} \label{4h1}
\delta> \max_{\{j\colon |\lambda_j|< \mu\}} |\lambda_j|
\quad\text{and}\quad
\delta>  |\rho|^2\mu.
\end{equation}
\end{theorem}
The proof is based on results of Agarwal and Pituk~\cite{Agarwal2007} as well as Bodine and 
Lutz~\cite{Bodine2009}; see Lemma~\ref{lemma-ap} below.
In the case that $K=1$ and thus $j_0=0$ and $j_1=m$ we can apply their results directly,
but the general case requires some extensions of their arguments; see Lemma~\ref{lemma-ap2}.

\begin{theorem} \label{zeros}
Let $f$ be a transcendental entire solution of \eqref{101} with growth given by~\eqref{104}.
Let $\mu$ be as in Theorem~{\rm\ref{coeff}} and suppose that the 
roots of the characteristic polynomial $P_k$ associated to $\sigma_k$
which have modulus $\mu$ differ only by roots of unity.
Then the zeros of $f$ are asymptotic to finitely many geometric progressions. 

More precisely, if $L\in\N$ with
$\lambda_i^{LM}=\lambda_j^{LM}$ whenever $|\lambda_i|=|\lambda_j|=\mu$, then
the set of zeros of $f$ can be written in the form
\begin{equation} \label{107a2}
\left\{z_{n,\nu}\colon\nu\in\{1,2,\dots,LM\},n\in\N\right\}
\end{equation}
such that for each $\nu$ there exists $A_\nu\in\C\setminus\{0\}$ with 
\begin{equation} \label{107a1}
z_{n,\nu}\sim A_\nu q^{-LN n}
\end{equation}
 as $n\to\infty$.
If, in addition, all roots of $P_k$ of modulus $\mu$ are simple, then~\eqref{107a1} can
be improved to
\begin{equation} \label{107a}
z_{n,\nu}=A_\nu q^{-LN n}
\left(1+\OO\!\left(\!\left(\frac{\delta}{\mu} \right)^{\!n}\right)\!\right),
\end{equation}
for any $\delta$ satisfying~\eqref{4h1}.
\end{theorem}
In Section~\ref{order2} we will discuss equations of order~$2$; that is, we will assume
that $m=2$ in~\eqref{101}.  Using~\cite[Example~1 and Theorem~3]{Bergweiler2003}
we will see that the conclusion of Theorem~\ref{zeros} need not hold if
the hypothesis on the zeros of $P_k$ which have modulus $\mu$ is not satisfied.
In fact, in the case that $a_j(z)=a_{j,j}z^j$ for $j\in\{0,1,2\}$,
this hypothesis is not only sufficient but also necessary in order to conclude that 
the zeros of every entire solution are asymptotic to a finite number of geometric 
progressions.

We will also see that~\eqref{107a} need not hold if $P_k$ has a multiple root of modulus~$\mu$;
see Case 2 in  Section~\ref{order2}.

We consider what Theorems~\ref{coeff} and~\ref{zeros} say if $I_k=\{j_{k-1},j_k\}$ 
as in~\cite{Bergweiler2003}. In this case we have
\begin{equation} \label{107x}
P_k(z)=c_{j_k,d(j_k)}z^M+c_{j_{k-1},d(j_{k-1})}.
\end{equation}
Choosing $\xi$ with $\xi^M=-c_{j_{k-1},d(j_{k-1})}/c_{j_k,d(j_k)}$ and putting
$\omega_j:=\exp(2\pi i j/M)$ we may take $\lambda_j=\omega_j\xi$.
Note that all $\lambda_j$ are simple roots so that the $Q_j$ in~\eqref{asymp-coeff} are constant.
Since $\lambda_j^n=\xi^n \omega_j^r$ if $n\equiv r  \pmod M$, we see that~\eqref{asymp-coeff} reduces 
to~\eqref{106} in this case, except for a slightly better error term in~\eqref{106}:
instead of the term $\OO(|\rho|^{2n})$ in~\eqref{106} we now only have $\OO(\tau^{2n})$ for any
$\tau>|\rho|$.

Note also that in this case the $\lambda_j$ differ only by roots of unity.
In fact, the hypotheses of Theorem~\ref{zeros} are satisfied with $L=1$. Hence~\eqref{107a} 
takes the form~\eqref{107b}, again apart from the error term. It is possible, however, to
improve the error bound in~\eqref{107a}; see Remark~\ref{rem1}. 

To summarize, Theorems~\ref{coeff} and~\ref{zeros} recover the results of~\cite{Bergweiler2003},
except for a slightly weaker error term.

Theorem~{\rm\ref{coeff}} also allows to sharpen~\eqref{104} as well as~\cite[Theorem 4.8]{Ramis1992}.
\begin{theorem} \label{growth}
Let $f$ be a transcendental entire solution of \eqref{101} with growth given by~\eqref{104}.
Let $\mu$ be as in Theorem~{\rm\ref{coeff}} and let
$\kappa$ be the maximum of the degrees of the polynomials $Q_j$ appearing in~\eqref{asymp-coeff}.
Then
\begin{equation} \label{104b}
\log M(r,f)= \frac{\sigma_k}{-2\log |q|}\left(\log \mu r\right)^2+\kappa\log\log r + \OO(1).
\end{equation}
\end{theorem}
We mention that the term $\OO(1)$ in~\eqref{104b} cannot be replaced by $c+o(1)$ 
with a constant~$c$; see Remark~\ref{remOo}.
\begin{ack}
A part of this paper was written during a stay at the Shanghai Center for Mathematical Sciences (SCMS).
I~thank the SCMS for the hospitality.
I~am also grateful to the referee for many helpful suggestions.
\end{ack}

\section{Proof of Theorem~\ref{coeff}} \label{proof-coeff}
A linear recurrence relation
\begin{equation} \label{4a}
x(n)=a_1 x(n-1)+a_2 x(n-2)+\ldots +a_{d} x(n-d)
\end{equation}
with constant coefficients $a_1,\ldots,a_d\in\C$ can be solved by considering the 
associated characteristic polynomial
\begin{equation} \label{4b}
P(\lambda)=\lambda^d -a_1 \lambda^{d-1}-a_2 \lambda^{d-2}-\ldots -a_{d}.
\end{equation}
Denoting by $\lambda_1,\dots,\lambda_k$ the zeros of $P$ and by $m_1,\dots,m_k$ their multiplicities,
the general solution of~\eqref{4a} is given by
\begin{equation} \label{4c}
x(n)=\sum_{j=1}^k Q_j(n) \lambda_j^n,
\end{equation}
where $Q_j$ is a polynomial of degree at most $m_j-1$.

It is plausible that if the coefficients $b_j(n)$ of the recurrence relation
\begin{equation} \label{4d}
y(n)=b_1(n) y(n-1)+b_2(n) y(n-2)+\ldots +b_{d}(n) y(n-d)
\end{equation}
satisfy 
\begin{equation} \label{4e}
\lim_{n\to\infty} b_j(n)=a_j ,
\end{equation}
then a solution of~\eqref{4d} will behave asymptotically like a solution of~\eqref{4a}.
Classical results of Poincar\'e~\cite{Poincare} and Perron~\cite[\S 5]{Perron} say that this is indeed
the case under suitable hypotheses.

We shall use the following result of Agarwal and Pituk~\cite[Theorem~2.3]{Agarwal2007} 
as well as Bodine and Lutz~\cite[Corollary~2]{Bodine2009}
which addresses the case where the convergence in~\eqref{4e} is exponentially fast.
\begin{lemma} \label{lemma-ap}
Let $0<\tau<1$ and suppose that the coefficients $b_j(n)$ in~\eqref{4d} satisfy 
\begin{equation} \label{4f}
b_j(n)=a_j +\OO\!\left(\tau^n\right)
\end{equation}
as $n\to\infty$, with $a_1,\dots,a_d\in\C$ and $a_d\neq 0$.
Let $P$ be the characteristic polynomial of~\eqref{4a}, as defined in~\eqref{4b}, 
and let $\lambda_1,\dots,\lambda_k$ the its zeros, with multiplicities $m_1,\dots,m_k$.

Let $y(n)$ be a solution of~\eqref{4d}. Then, unless $y(n)=0$ for all large $n$, there exist
$\mu\in\{|\lambda_1|,\dots,|\lambda_k|\}$ such that for every $\delta\in (0,\mu)$ satisfying
\begin{equation} \label{4h}
\delta> \max_{\{j\colon |\lambda_j|< \mu\}} |\lambda_j|
\quad\text{and}\quad \delta>  \tau\mu
\end{equation}
we have
\begin{equation} \label{4g}
y(n)=\sum_{\{j\colon |\lambda_j|=\mu\}} Q_j(n)\lambda_j^n +\OO\!\left(\delta^n\right)
\end{equation}
as $n\to\infty$, with polynomials $Q_j$ of degree at most $m_j-1$, which do not all vanish
identically.
\end{lemma}
Agarwal and Pituk~\cite[Theorem~2.3]{Agarwal2007} proved this with $\delta=\mu-\varepsilon$ for some $\varepsilon>0$.
Bodine and Lutz~\cite{Bodine2009} gave a different proof of this result, showing
that one may take any $\delta$ satisfying~\eqref{4h}.
In fact, they gave a corresponding result for first order systems of linear difference equations.

We can rewrite the recurrence relation~\eqref{4d} in the form
\begin{equation} \label{4i}
\sum_{j=0}^d b_j(n) y(n-j)=0. 
\end{equation}
To obtain~\eqref{4i} from~\eqref{4d} one simply puts $b_0(n)\equiv -1$.
In the opposite direction one divides~\eqref{4i} by $b_0(n)$, provided this term is non-zero.

Thus Lemma~\ref{lemma-ap} can also be applied to equations given in 
the form~\eqref{4i}, provided that~\eqref{4f} also holds for $j=0$,
with $a_0\neq 0$. We thus require that the numbers $a_j$ given by~\eqref{4f} 
satisfy $a_0\neq 0$ and $a_d\neq 0$.

One consequence of~\eqref{4g} is that
\begin{equation} \label{4l1}
-\infty< \limsup_{n\to\infty}\frac{\log |y(n)|}{n}<\infty.
\end{equation}
In fact, the upper limit in~\eqref{4l1} is given by $\log\mu$.

We note that~\eqref{4l1} need not hold if $a_0= 0$ or $a_d= 0$.
To see this, consider for example the equation
\begin{equation} \label{4j}
\frac{1}{2^{2n\pm 1}} y(n) - \left( 1 +\frac{1}{2^{4n}}\right) y(n-1)+
\frac{1}{2^{2n\pm 1}} y(n-2)=0 .
\end{equation}
This equation is of the form~\eqref{4i} with $d=2$, and~\eqref{4f} is satisfied 
with $a_0=a_2=0$, $a_1=-1$ and $\tau=1/4$.
It has the solutions 
\begin{equation} \label{4j0}
y(n)=2^{\pm (n+1)^2} .
\end{equation}
For these solutions we have $(\log|y(n)|)/n\to\pm\infty$ so that~\eqref{4l1} does not hold.
We conclude that if the conditions $a_0\neq 0$ and $a_d\neq 0$ are not satisfied,
then~\eqref{4i} may have solutions for which~\eqref{4l1} does not hold. In particular,
these solutions are not of the form~\eqref{4g}.

The following result addresses the case that $a_0=0$ or $a_d=0$. 
Essentially, it says that a solution of~\eqref{4i} is of the form~\eqref{4g}, if 
the condition~\eqref{4l1} is satisfied. 
In fact, a weaker growth restriction than~\eqref{4l1} will suffice.

This result will be the main tool in the proof of Theorem~\ref{coeff}.
Its proof will use Lemma~\ref{lemma-ap}, by reducing the equation considered to a type where
this lemma is applicable.
\begin{lemma} \label{lemma-ap2}
Let $0<\tau<1$ and suppose that the coefficients $b_j(n)$ in~\eqref{4i} satisfy~\eqref{4f}
as $n\to\infty$. Let $s,t\in\{0,1,\dots,d\}$ with $s< t$ such that $a_s\neq 0$, $a_t\neq 0$ and $a_j=0$ for
$0\leq j<s$ and $t<j\leq d$.

Let 
\begin{equation} \label{4k}
P(z)=\sum_{j=0}^{t-s} a_{s+j} z^{j}
\end{equation}
and let $\lambda_1,\dots,\lambda_k$ the the zeros of $P$, with multiplicities $m_1,\dots,m_k$.
Let $y(n)$ be a solution of~\eqref{4i} satisfying
\begin{equation} \label{4l}
\limsup_{n\to\infty}\frac{\log |y(n)|}{n^2}=0.
\end{equation}
Then there exists $\mu\in\{|\lambda_1|,\dots,|\lambda_k|\}$ such that~\eqref{4g} holds for every $\delta$ satisfying~\eqref{4h},
with polynomials $Q_j$ of degree at most $m_j-1$, which do not all vanish identically.
\end{lemma}
\begin{rem} 
The example~\eqref{4j} with solutions~\eqref{4j0} shows that the condition~\eqref{4l} cannot be omitted
and is in fact sharp.
\end{rem} 
\begin{proof}[Proof of Lemma~\ref{lemma-ap2}] 
First we show that given $K>1$ we have
\begin{equation}
\label{402}
|y(n)|\leq K^n\max_{1\leq j\leq d-s}|y(n-j)|
\end{equation}
for all large~$n$.
The argument for this is in part similar to that in~\cite[Lemma~2]{Bergweiler2003}.
We consider
\begin{equation} \label{4m}
I:=\left\{n>d-s\colon |y(n)|> K^n\max_{1\leq j\leq d-s}|y(n-j)| \right\}
\end{equation}
and have to show that $I$ is bounded. In order to do so we may assume that $K^d\tau<1$.

First we show that if $m\in I$ and $p\in\N$ such that $\{m+1,\dots,m+p\}\cap I=\emptyset$, then 
\begin{equation} \label{4r0}
|y(m+j)| \leq K^{jm+j(j+1)/2} |y(m)|
\end{equation}
for $1\leq j\leq p$. 
In order to see this, we note that since $m\in I$ and $m+1\notin I$ we have
\begin{equation} \label{4o0}
|y(m+1)|\leq K^{m+1} \max_{1\leq j\leq d-s}|y(m+1-j)| = K^{m+1}|y(m)|.
\end{equation}
Next, if also $m+2\notin I$, then
\begin{equation} \label{4q0}
\begin{aligned}
|y(m+2)|
&\leq K^{m+2} \max_{1\leq j\leq d-s}|y(m+2-j)|
\\ &
= K^{m+2}\max\{ |y(m+1)|,|y(m)|\}
\leq K^{2m+3} |y(m)|.
\end{aligned}
\end{equation}
Induction shows that~\eqref{4r0} holds for $1\leq j\leq p$.

Suppose now that $I$ is unbounded.
We show that there exist arbitrarily large $m\in I$ such that~\eqref{4r0} holds for $1\leq j\leq s$.
Suppose that this is not the case. Then for every large $m\in I$ there 
exists $l\in\{1,\dots,s\}$ such that~\eqref{4r0} holds for $1\leq j\leq l-1$, but does 
not hold for $j=l$; that is,
\begin{equation} \label{4r1}
|y(m+l)| > K^{lm+l(l+1)/2} |y(m)|.
\end{equation}
Since~\eqref{4r0} holds for $1\leq j\leq l-1$ we find that 
\begin{equation} \label{4r2}
\begin{aligned}
|y(m+l)| 
&> K^{lm+l(l+1)/2} K^{-jm-j(j+1)/2}  |y(m+j)|
\\ &
\geq K^{lm+l(l+1)/2} K^{-(l-1)m-(l-1)l/2}  |y(m+j)|
= K^{m+l}  |y(m+j)|
\end{aligned}
\end{equation}
for $1\leq j\leq l-1$. Together with~\eqref{4r1} and the assumption that $m\in I$
this yields that $m+l\in I$.
Thus there exists an increasing sequence $(m_k)$ in $I$ such that $l_k:=m_{k+1}-m_k\leq s$ and
\begin{equation} \label{4r3}
|y(m_{k+1})| > K^{l_k m_k +l_k(l_k+1)/2}  |y(m_k)|
\geq K^{m_k}  |y(m_k)|
\geq K^{m_1+\dots +m_k}  |y(m_1)|.
\end{equation}
Since 
\begin{equation} \label{4r4}
m_{k+1}=m_1+\sum_{j=1}^kl_j\leq m_1+ks
\quad\text{and}\quad
\sum_{j=1}^k m_j\geq \sum_{j=1}^k j=\frac{k(k+1)}{2}\geq \frac{k^2}{2}
\end{equation}
this yields that 
\begin{equation} \label{4r5}
\begin{aligned}
\log |y(m_{k+1})| 
&> \frac{k^2}{2} \log K +\log |y(m_{1})|
\\ &
\geq \frac{(m_{k+1}-m_1)^2}{2s^2} \log K +\log |y(m_{1})|
 \geq  \frac{m_{k+1}^2}{3s^2}
\end{aligned}
\end{equation}
for large $k$, contradicting~\eqref{4l}.
We have shown, still assuming that $I$ is unbounded, that there exist arbitrarily large $m\in I$
such that~\eqref{4r0} holds for $1\leq j\leq s$.

For such $m$ we put $n=m+s$ so that $m=n-s$. Then we have $n-s\in I$ and~\eqref{4r0} yields that
\begin{equation} \label{4s0}
\begin{aligned}
|y(n-s+j)| 
&\leq K^{j(n-s)+j(j+1)/2} |y(n-s)|
\\ &
\leq K^{s(n-s)+s(s+1)/2} |y(n-s)|\leq K^{sn}|y(n-s)|
\end{aligned}
\end{equation}
for $1\leq j\leq s$.
Replacing $j$ by $s-j$ we thus see that
\begin{equation} \label{4s}
|y(n-j)| \leq K^{sn} |y(n-s)| \quad\text{for}\ 0\leq j\leq s-1.
\end{equation}
On the other hand, since $n-s\in I$, we have 
\begin{equation} \label{4s1}
|y(n-j)| < K^{s-n} |y(n-s)|
\quad\text{for}\ s+1\leq j\leq d.
\end{equation}
Together with~\eqref{4i} the last two inequalities imply that
\begin{equation} \label{4s2}
\begin{aligned} 
|b_s(n)y(n-s)|
& \leq \sum_{j=0}^{s-1} |b_j(n) y(n-j)|
+ \sum_{j=s+1}^{d} |b_j(n) y(n-j)|
\\ &
\leq |y(n-s)| \left(K^{sn} \sum_{j=0}^{s-1} |b_j(n)| + K^{s-n}\sum_{j=s+1}^{d} |b_j(n)|\right).
\end{aligned} 
\end{equation}
Since $n-s\in I$ we have $y(n-s)\neq 0$ and may thus divide the last inequality by $|y(n-s)|$.
Choosing $C_1$ such that $|b_j(n)-a_j|\leq C_1\tau^n$ for all $j\in\{0,1,\dots,d\}$ and 
all $n\in\N$, and noting that $a_j=0$ for $0\leq j\leq s-1$,
we conclude that 
\begin{equation} \label{4t}
\begin{aligned} 
|a_s|-C_1\tau^n
&\leq  K^{sn}sC_1\tau^n  + K^{s-n}\sum_{j=s+1}^{d} \left( |a_j|  +C_1\tau^n \right)
\\ &
=  sC_1\left(K^{s}\tau\right)^n  + K^{s-n}\sum_{j=s+1}^{d} |a_j|  +K^{s-n}(d-s)C_1\tau^n .
\end{aligned} 
\end{equation}
Since  $a_s\neq 0$  and $K^s\tau<K^d\tau<1$, this is a contradiction for large~$n$.
Hence $I$ is bounded so that~\eqref{402} holds for all large~$n$.

It follows from~\eqref{402} that 
\begin{equation} \label{4u}
|y(n-s+1)|\leq K^{n-s+1}\max_{1\leq j\leq d-s}|y(n-s+1-j)|
= K^{n-s+1}\max_{s\leq j\leq d-1}|y(n-j)|
\end{equation}
for large $n$. As in the proof of~\eqref{4r0} we can use induction to show
that
\begin{equation} \label{4v3}
\begin{aligned} 
|y(n-s+i)| 
&\leq K^{i(n-s)+i(i+1)/2} \max_{s\leq j\leq d-1}|y(n-j)|
\\ &
\leq K^{s(n-s)+s(s+1)/2} \max_{s\leq j\leq d-1}|y(n-j)|
\\ &
\leq K^{sn} \max_{s\leq j\leq d-1}|y(n-j)|
\quad\text{for}\ 1\leq i\leq s.
\end{aligned} 
\end{equation}
Replacing $i$ by $s-i$ we thus find  that
\begin{equation} \label{4v}
|y(n-i)| \leq K^{sn} \max_{s\leq j\leq d-1}|y(n-j)|
\quad\text{for}\ 0\leq i\leq s-1.
\end{equation}
Since $a_i=0$ for $0\leq i\leq s-1$ this yields that
\begin{equation} \label{4v0}
\begin{aligned}
|b_i(n)y(n-i)| 
&\leq C_1 \tau^n  K^{sn} \max_{s\leq j\leq d-1}|y(n-j)|
\\ &
= C_1  \left(K^s\tau\right)^{n} \max_{s\leq j\leq d-1}|y(n-j)|
\quad\text{for}\ 0\leq i\leq s-1.
\end{aligned}
\end{equation}
This allows to rewrite~\eqref{4i} in the form 
\begin{equation} \label{4i1}
\sum_{j=s}^d b_j^*(n) y(n-j)=0 ,
\end{equation}
where the $b_j^*$ satisfy 
$b_j^*(n)=a_j +\OO((K^s\tau)^n)$, say
\begin{equation} \label{4f1}
|b_j^*(n)-a_j|\leq C_2 \left(K^s\tau\right)^n,
\end{equation}
instead of~\eqref{4f}. Thus we have disposed of the first $s$ terms in the sum 
in~\eqref{4i}, at the expense of slightly increasing the error term in~\eqref{4f}.

Our next aim is to dispose of the last $d-t$ terms.
This is achieved by a very similar reasoning, but for the convenience of the reader 
we include the argument.
First, analogously to~\eqref{402}, we show that given $K>1$ we have
\begin{equation}
\label{402a}
|y(n)|\leq K^n\max_{1\leq j\leq t-s}|y(n+j)|
\end{equation}
for large~$n$. We have to show that
\begin{equation} \label{4m1}
J:=\left\{n>t-s\colon |y(n)|> K^n\max_{1\leq j\leq t-s}|y(n+j)| \right\}
\end{equation}
is bounded.
Analogously to~\eqref{4r0} we show first that if $m\in J$ and $p\in\N$ with $p<m$
such that $\{m-1,\dots,m-p\}\cap J=\emptyset$, then
\begin{equation} \label{4x0}
|y(m-j)| \leq K^{jm-j(j+1)/2} |y(m)|
\end{equation}
for $1\leq j\leq p$.
Assuming that $J$ is unbounded we can,
similarly as before, use~\eqref{4l} to show that there are arbitrarily large $m\in J$ such
that~\eqref{4x0} holds for $1\leq j\leq d-t$.
For such $m$ we now put $n=m+t$ so that $m=n-t$. 
Since~\eqref{4x0} holds for $1\leq j\leq d-t$ we obtain 
\begin{equation} \label{4x1}
|y(n-t-j)| \leq K^{j(n-t)-j(j+1)/2} |y(n-t)| \leq K^{(d-t)n}|y(n-t)| 
\end{equation}
for $1\leq j\leq d-t$.
Replacing $j$ by $j-t$ yields
\begin{equation} \label{4x2}
|y(n-j)| \leq K^{(d-t)n}|y(n-t)| \quad\text{for}\ t+1\leq j\leq d.
\end{equation}
On the other hand, since $n-t\in J$, we also have
\begin{equation} \label{4x3}
|y(n-j)| < K^{t-n}|y(n-t)| \quad\text{for}\ s\leq j\leq t-1.
\end{equation}
Inserting the last two inequalities into~\eqref{4i1} yields that
\begin{equation} \label{4y2}
\begin{aligned} 
|b_t^*(n)y(n-t)|
& \leq \sum_{j=s}^{t-1} |b_j^*(n) y(n-j)|
+ \sum_{j=t+1}^{d} |b_j^*(n) y(n-j)|
\\ &
\leq |y(n-t)| \left(K^{t-n} \sum_{j=s}^{t-1} |b_j^*(n)| + K^{(d-t)n}\sum_{j=t+1}^{d} |b_j^*(n)|\right).
\end{aligned} 
\end{equation}
Noting that $a_j=0$ for $t+1\leq j\leq d$ and that $y(n-t)\neq 0$ since $n-t\in J$ we deduce
from~\eqref{4f1} that
\begin{equation} \label{4y3}
\begin{aligned} 
|a_t|-C_2(K^s\tau)^n
&\leq  K^{t-n}\sum_{j=s}^{t-1} \left( |a_j| + C_2(K^s\tau)^n\right) 
 +K^{(d-t)n}(d-t)C_2(K^s\tau)^n 
\\ &
\leq  K^{t-n}\sum_{j=s}^{t-1} |a_j| + K^{t-n}(t-s)C_2(K^s\tau)^n 
 +(d-t)C_2(K^{d-t+s}\tau)^n .
\end{aligned} 
\end{equation}
Since $a_t\neq 0$ and $K^{d-t+s}\tau<K^d\tau<1$, this is a contradiction for large~$n$.
Hence $J$ is bounded, meaning that~\eqref{402a} holds for all large~$n$.

In the same way that we used~\eqref{402} to obtain~\eqref{4v} we can now use~\eqref{402a} to
obtain
\begin{equation} \label{4v1}
|y(n-i)| \leq  K^{(d-t)n} \max_{s+1\leq j\leq t}|y(n-j)|
\quad\text{for}\ t+1\leq i\leq d.
\end{equation}
And similarly as before we can use this to rewrite~\eqref{4i1} in the form
\begin{equation} \label{4i2}
\sum_{j=s}^t b_j^{**}(n) y(n-j)=0
\end{equation}
where the $b_j^{**}$ satisfy
\begin{equation} \label{4f2}
b_j^{**}(n)=a_j +\OO\!\left(\!\left(K^{d-t}K^s\tau\right)^n\right).
\end{equation}
Note that~\eqref{4i2} is equivalent to
\begin{equation} \label{4i3}
\sum_{j=0}^{t-s} \widetilde{b}_{j}(n) y(n-j)=0,
\end{equation}
with
\begin{equation} \label{4f3}
\widetilde{b}_{j}(n):=b_{s+j}^{**}(n+s)=a_{s+j} +\OO\!\left(\!\left(K^{d-t+s}\tau\right)^n\right).
\end{equation}
Since $a_s\neq 0$ and $a_t\neq 0$ we can apply Lemma~\ref{lemma-ap} to the equation~\eqref{4i3}.
Recalling that the equations~\eqref{4d} and~\eqref{4i} are equivalent and noting that the characteristic 
polynomial associated to the equation~\eqref{4i3} is given by~\eqref{4k},
we conclude that $y(n)$ satisfies~\eqref{4g}.
Instead of the condition $\delta>\tau\mu$ as in~\eqref{4h} we first
obtain only $\delta>K^{d-t+s}\tau\mu$. However, since we can take any $K>1$ in this condition,
we are again led to $\delta>\tau\mu$.
\end{proof}
As in~\cite{Bergweiler2003} we will also use the following
lemma, which is a special case of a result of Juneja, Kapoor and Bajpai~\cite[Theorem~1]{Juneja1977}.
\begin{lemma} \label{la1}
Let $f$ be a transcendental entire function with Taylor series given by~\eqref{105}.
Then
\begin{equation} \label{301x}
\limsup_{r\to\infty} \frac{\log M(r,f)}{(\log r)^2}
=\frac14 \limsup_{n\to\infty} \frac{n^2}{-\log|\alpha_n|}.
\end{equation}
\end{lemma}
Here we put $-\log|\alpha_n|=\infty$ and thus ${n^2}/(-\log|\alpha_n|)=0$  if $\alpha_n=0$.

\begin{proof}[Proof of Theorem~\ref{coeff}]
From the definition of $\sigma_k$ and $I_k$ we have
\begin{equation} \label{301}
d(j) \le d(j_k)+\sigma_{k}(j-j_{k})=d(j_{k-1})+\sigma_{k}(j-j_{k-1})
\end{equation}
if $j\in \{0,1,\dots,m\}$, with equality if and only if $j\in I_k$.
Recalling that $\sigma_k =M/N$ by~\eqref{MN} we may write~\eqref{301} in the form
\begin{equation} \label{301b}
N d(j)- M j\leq N d(j_k)-M j_k= N d(j_{k-1})-M j_{k-1}=:L,
\end{equation}
with strict inequality if $j\notin I_k$. Hence $L+M j-N d(j)\geq 0$,
with strict inequality if $j\notin I_k$.  
For $0\leq j\leq m$ and $0\leq i\leq d(j)$ we thus have
\begin{equation} \label{301d}
L+M j-N i\geq 0,
\end{equation}
with equality if and  only if $(j,i) = (j,d(j))$ for some $j\in I_k$. 

Recall the definition of the $a_{j,i}$ as the coefficients of the $a_j$ 
given in~\eqref{a_j} and put $a_{j,i}=0$ for $i>d(j)$.
Then the Taylor series expansion of the left hand side of~\eqref{101} is given by
\begin{equation} \label{comp-coeff}
\begin{aligned} 
\sum_{j=0}^m a_j(z)f(q^jz)
& =
\sum_{j=0}^m \left( 
\left( \sum_{i=0}^\infty a_{j,i} z^i \right) \cdot
\left( \sum_{\nu=0}^\infty \alpha_{\nu} q^{j\nu} z^\nu \right)
\right)
\\ &
=\sum_{j=0}^m 
\sum_{n=0}^\infty \left( 
\sum_{i=0}^n a_{j,i} \alpha_{n-i} q^{j(n-i)}
\right) z^n
\\ &
=\sum_{n=0}^\infty \left( 
\sum_{i=0}^n \alpha_{n-i}
\sum_{j=0}^m a_{j,i} q^{j(n-i)}
\right) z^n .
\end{aligned} 
\end{equation}
Comparing coefficients in~\eqref{101} we find with
\begin{equation} \label{defd}
d:=\max\limits_{j\in\{0,\dots,m\}} d(j)
\end{equation}
that if $n>\max\{d,\deg(b)\}$, then
\begin{equation} \label{301e}
\sum_{i=0}^d\alpha_{n-i}\sum_{j=0}^m a_{j,i} q^{j(n-i)}=0.
\end{equation}
With 
\begin{equation} \label{defyn}
y(n):=\rho^{-N n^2}\alpha_n
\end{equation}
and $c_{j,i}=a_{j,i}\rho^{N i^2 -2M ji}$ as defined in~\eqref{defc} 
we deduce, using~\eqref{defrho}, that 
\begin{equation} \label{recy}
0= \sum_{i=0}^d \rho^{N(n-i)^2} y(n-i)\sum_{j=0}^m a_{j,i} q^{j(n-i)}
= \rho^{N n^2} \sum_{i=0}^d y(n-i)\sum_{j=0}^m c_{j,i} \rho^{(M j-N i)2n}.
\end{equation}
Multiplying this equation by $\rho^{-Nn^2+2Ln}$  we obtain
\begin{equation} \label{recbeta}
\sum_{i=0}^d B_i(n)y(n-i) =0.
\end{equation}
with
\begin{equation} \label{defB}
B_i(n) :=\rho^{2Ln} \sum_{j=0}^m c_{j,i} \rho^{(M j-N i)2n}=
\sum_{j=0}^m c_{j,i} \rho^{(L+M j-N i)2n}.
\end{equation}
By~\eqref{301d} we have
\begin{equation} \label{asymptBi}
B_i(n)  =\begin{cases}
c_{j,d(j)} +\OO\!\left(|\rho|^{2n}\right) 
& \text{if } j\in I_k \text{ and }i=d(j),\\[2mm]
\OO\!\left(|\rho|^{2n}\right) 
& \text{otherwise.}
\end{cases}
\end{equation}

The recurrence relation~\eqref{recbeta} thus is an equation of the form~\eqref{4i},
with the equation~\eqref{asymptBi} 
for the coefficients being the equivalent of~\eqref{4f}. 
The conclusion will follow from Lemma~\ref{lemma-ap2}, once we 
have checked that it is applicable to this equation.
It thus remains to show that~\eqref{4l} holds.

In order to do so we use Lemma~\ref{la1} which together with~\eqref{104} and~\eqref{MN} yields that
\begin{equation}\label{yn1}
\frac14 \limsup_{n\to\infty} \frac{n^2}{-\log|\alpha_n|}
=\frac{\sigma_k}{-2\log |q|}
=-\frac{M}{2N \log |q|}.
\end{equation}
Thus
\begin{equation}\label{yn2}
\limsup_{n\to\infty}
\frac{\log|\alpha_n|}{n^2}
= \frac{N \log |q|}{2M}.
\end{equation}
Since, by~\eqref{defyn} and~\eqref{defrho},
\begin{equation}\label{yn}
\log|y(n)|
=-Nn^2\log|\rho|+\log|\alpha_n|
=-\frac{N\log|q|}{2M}n^2+\log|\alpha_n| ,
\end{equation}
we conclude that~\eqref{4l} holds.
\end{proof}

\section{Proof of Theorem~\ref{growth}} \label{proof-growth}
The theta function is defined by
\begin{equation} \label{501}
\theta(z,q):=\sum_{n=-\infty}^\infty q^{n^2}z^n.
\end{equation}
The series converges for $|q|<1$ and $z\in\C\setminus\{0\}$.
It satisfies the functional equation 
\begin{equation} \label{equationtheta}
\theta(z,q)=qz\theta\!\left(q^2z,q\right)
\end{equation}
and Jacobi's triple product identity (see, e.g., \cite[Theorem~2.8]{Andrews1976}) 
\begin{equation} \label{jacobi}
\theta(z,q) = \prod_{n=1}^\infty  \left\{\left(1-q^{2n}\right) \left(1+\frac{q^{2n-1}}{z}\right)
\left(1+q^{2n-1}z\right) \right\}.
\end{equation}

More generally, we will consider for $\kappa \in\R$ the function 
\begin{equation} \label{8a}
\theta_\kappa (z,q):=\sum_{n=-\infty}^\infty n^\kappa  q^{n^2}z^n .
\end{equation}
It is easy to see that this series also 
converges for $|q|<1$ and $z\in\C\setminus\{0\}$.
\begin{lemma} \label{la10}
Let $0<q<1$ and $\kappa \in\R$. Then
\begin{equation} \label{8b}
\log M(r, \theta_\kappa (\cdot,q))=\frac{(\log r)^2}{-4\log q}+\kappa \log\log r+\OO(1)
\end{equation}
as $r\to\infty$.
\end{lemma}
\begin{proof}
Write $r=q^{-2(m+\alpha)}$ with $m\in\Z$ and $0\leq \alpha<1$.
Then
\begin{equation} \label{8g}
\begin{aligned}
M(r,\theta_\kappa (\cdot,q))
&
=\sum_{n=-\infty}^\infty n^\kappa  q^{n^2}r^n
=\sum_{n=-\infty}^\infty (n+m)^\kappa  q^{(n+m)^2}r^{n+m}
\\ &
=q^{m^2}r^{m}m^\kappa \sum_{n=-\infty}^\infty \left(1+\frac{n}{m}\right)^\kappa  q^{n^2+2nm} r^{n}
\\ &
=q^{-m^2-2\alpha m} m^\kappa 
\sum_{n=-\infty}^\infty \left(1+\frac{n}{m}\right)^\kappa  q^{n^2}q^{-2\alpha n}
\end{aligned}
\end{equation}
so that
\begin{equation} \label{8h}
\begin{aligned}
M\!\left(r,\theta_\kappa (\cdot,q)\right)
& \sim
q^{-m^2-2\alpha m} m^\kappa 
\sum_{n=-\infty}^\infty q^{n^2} q^{-2\alpha n}
=\exp\!\left( \frac{(\log r)^2}{-4\log q}\right) m^\kappa  q^{\alpha^2}
\theta\!\left(q^{-2\alpha},q\right)
\end{aligned}
\end{equation}
as $r\to\infty$ and hence $m\to\infty$.
Since 
\begin{equation} \label{8h1}
m= \frac{\log r}{-2\log q} -\alpha\sim \frac{\log r}{-2\log q}
\end{equation}
 we have 
\begin{equation} \label{8i}
\log(m^\kappa )=\kappa \log m=\kappa  \log\log r-\kappa\log(-2\log q)+o(1)
\end{equation}
and thus
\begin{equation} \label{8b1}
\begin{aligned}
\log M(r, \theta_\kappa (\cdot,q))
&=\frac{(\log r)^2}{-4\log q}+\kappa \log\log r
-\kappa \log(-2\log q) 
\\ &
\qquad
+\log\!\left( q^{\alpha^2}
\theta(q^{-2\alpha},q)\right)+o(1)
\end{aligned}
\end{equation}
as $r\to\infty$ so that the conclusion follows.
\end{proof}
\begin{remark} \label{remOo}
The term $q^{\alpha^2} \theta(q^{-2\alpha},q)$ occurring in~\eqref{8b1}
 shows that the term $\OO(1)$ in~\eqref{8b} cannot be replaced by $c+o(1)$
for some constant~$c$. This also shows that the term $\OO(1)$ in~\eqref{104b} 
cannot be improved to $c+o(1)$ with a constant~$c$.

It turns out that this is not a special property of solutions of~\eqref{101}.
In fact, a recent result of Hilberdink~\cite[Theorem~ 1.1]{Hilberdink2020} says
that if $h\colon [x_0,\infty)\to(0,\infty)$ is a twice continuously differentiable 
function satisfying
$h''(x)>0$ for $x\geq x_0$, $h''(x)\to C$ as $x\to\infty$ for some $C\geq 0$ 
and $h'(x)\to \infty$ as $x\to\infty$, then there is no entire function $f$ 
such that 
\begin{equation} \label{8i0}
\log M(r,f) =h(\log r)+ o(1) 
\end{equation}
as $r\to\infty$.
Choosing
\begin{equation} \label{6m}
h(x)= \frac{\sigma_k}{-2\log |q|}(x+\log\mu)^2+\kappa\log x +c
\end{equation}
with a constant~$c$, we see that there is no entire function $f$
satisfying~\eqref{104b} with $\OO(1)$ replaced by $c+o(1)$.
\end{remark} 
We shall also need the following two lemmas.
\begin{lemma} \label{la8}
Let $w_1,\dots,w_n\in\C\setminus\{0\}$ be distinct.
Then there exists $\eta>0$ such that
\begin{equation} \label{6d}
\left\|\left( \sum_{j=1}^n w_j^0 x_j,\sum_{j=1}^n w_j^1 x_j,
\dots, \sum_{j=1}^n w_j^{n-1} x_j\right)\right\|_\infty
\geq \eta \|(x_1,x_2,\dots,x_n)\|_\infty
\end{equation}
for all $(x_1,\dots,x_n)\in\C^n$.
\end{lemma}
\begin{proof}
The vector on the left hand side is obtained by multiplying the one on
the right hand side by the Vandermonde matrix. The conclusion holds if
we take $1/\eta$ as the operator norm of the inverse of the Vandermonde matrix.
\end{proof}
\begin{lemma} \label{la11}
Let $A,B\colon [x_0,\infty)\to\R$ be convex. Suppose that 
$B$ is differentiable and satisfies 
\begin{equation} \label{Bx+1}
B'(x+1)= B'(x)+\OO(1)
\end{equation}
as $x\to\infty$.
Suppose also that there exists an increasing sequence $(x_k)$ in $[x_0,\infty)$ 
such that $x_k\to\infty$, $x_{k+1}=x_k+\OO(1)$ and $A(x_k)=B(x_k)+\OO(1)$ 
as $k\to\infty$. Then  $A(x)=B(x)+\OO(1)$ as $x\to\infty$.
\end{lemma}
\begin{proof}
Let $A'$ be the right derivative of $A$. This exists and is non-decreasing
since $A$ is convex.
We may assume that $x_{k+1}-x_k\geq 1$ for all $k$, since this may be achieved by passing
to a subsequence.  Then 
\begin{equation} \label{8f}
A'(x_k)\geq   \frac{A(x_k)-A(x_{k-1})}{x_k-x_{k-1}}
= \frac{B(x_k)-B(x_{k-1})}{x_k-x_{k-1}}+\OO(1)\geq B'(x_{k-1})+\OO(1).
\end{equation}
By~\eqref{Bx+1} we have
\begin{equation} \label{8f1}
B'(x_{k-1})=B'(x_k)+\OO(1).
\end{equation}
Thus $A'(x_k)\geq B'(x_k)+\OO(1)$. 

An analogous argument yields that $A'(x_k)\leq B'(x_k)+\OO(1)$.
Indeed, this follows from
\begin{equation} \label{8f2}
A'(x_k)\leq   \frac{A(x_{k+1})-A(x_{k})}{x_{k+1}-x_{k}}
= \frac{B(x_{k+1})-B(x_{k})}{x_{k+1}-x_{k}}+\OO(1)\leq B'(x_{k+1})+\OO(1)
\end{equation}
and~\eqref{8f1} with $k$ replaced by $k+1$.

Thus $A'(x_k)= B'(x_k)+\OO(1)$.
Since both $A'$ and $B'$ are non-decreasing this yields together with~\eqref{8f1} that $A'(x)=B'(x)+\OO(1)$
as $x\to\infty$.
For $x\geq x_1$ we choose $k$ such that $x_k\leq x<x_{k+1}$. Then
\begin{equation} \label{8e}
A(x)=A(x_k)+\int_{x_k}^x A'(t)dt 
=B(x_k)+\int_{x_k}^x B'(t)dt +\OO(1)
=B(x)+\OO(1)
\end{equation}
as $x\to\infty$.
\end{proof}
\begin{remark}
We do not really require that $B$ is differentiable. It suffices to assume that the right (or left)
derivative $B'$ of $B$ satisfies~\eqref{Bx+1}.
\end{remark}
\begin{proof}[Proof of Theorem~\ref{growth}]
By~\eqref{asymp-coeff}, there exist $C>0$ such that
\begin{equation} \label{6e}
\left|\alpha_n\rho^{-N n^2}\right| \leq Cn^\kappa\mu^n
\quad\text{for}\ n\in\N.
\end{equation}
It follows that
\begin{equation} \label{6e1}
\begin{aligned} 
M(r,f)
&\leq \sum_{n=0}^\infty |\alpha_n| r^n 
\leq |\alpha_0| +\sum_{n=1}^\infty Cn^\kappa\mu^n |\rho|^{N n^2} r^n 
\\ &
\leq |\alpha_0| +C \sum_{n=-\infty}^\infty n^\kappa |\rho|^{N n^2} (\mu r)^n 
=|\alpha_0|+CM\!\left(\mu r,\theta_\kappa\!\left(\cdot,|\rho|^N\right)\right).
\end{aligned} 
\end{equation}
Using~\eqref{defrho} and~\eqref{MN} we find that
\begin{equation} \label{6e2}
|\rho|^N=|q|^{N/(2M)}=|q|^{1/2\sigma_k}.
\end{equation}
Thus the last inequality can also be written as
\begin{equation} \label{6f}
M(r,f) \leq 
|\alpha_0|+ CM\!\left(\mu r,\theta_\kappa\!\left(\cdot,|q|^{1/2\sigma_k}\right)\right).
\end{equation}
Lemma~\ref{la10} now yields that
\begin{equation} \label{104b1}
\log M(r,f)\leq \frac{\sigma_k}{-2\log |q|}\left(\log \mu r\right)^2+\kappa\log\log r + \OO(1).
\end{equation}

In the opposite direction, let $s$ be the number of polynomials $Q_j$ which
have degree $\kappa$. Without loss of generality we may assume that $\deg Q_j=\kappa$ for
$1\leq j\leq s$, say $Q_j(z)\sim \gamma_j z^\kappa$ as $z\to\infty$, with $\gamma_j\neq 0$.
Let $n\in\N$. Applying Lemma~\ref{la8} with $w_j=\lambda_j$ and $x_j=\lambda_j^n \gamma_j$
we find that there exists $i\in \{0,1,\dots,s-1\}$ such that
\begin{equation} \label{6g}
\begin{aligned}
\left|\sum_{j=1}^s \lambda_j^{n+i}\gamma_j \right| 
&=
\left|\sum_{j=1}^s w_j^i x_j \right|
\geq \eta \|(x_1,\dots,x_s)\|_\infty
= \eta \|(\gamma_1,\dots,\gamma_s)\|_\infty \mu^n
\geq 2 c \mu^{n+i},
\end{aligned}
\end{equation}
with 
\begin{equation} \label{6gx}
c:=\frac12\eta \|(\gamma_1,\dots,\gamma_s)\|_\infty\min\{\mu^{1-s},\mu^{s-1}\}.
\end{equation}
It follows that there exists an increasing sequence $(n_l)$ 
satisfying $n_{l+1}\leq n_l+s$ such that
\begin{equation} \label{6g2}
\left|\sum_{j=1}^s \lambda_j^{n_l}\gamma_j \right| 
\geq 2 c \mu^{n_l}.
\end{equation}
By~\eqref{asymp-coeff} we have
\begin{equation} \label{6g1a}
\alpha_{n_l}\rho^{-N n_l^2}= n_l^\kappa \sum_{j=1}^s \lambda_j^{n_l}\gamma_j
+\OO\!\left(n_l^{\kappa-1}\mu^{n_l}\right).
\end{equation}
Together with~\eqref{6g2} this yields that
\begin{equation} \label{6g1}
\left|\alpha_{n_l}\rho^{-N n_l^2}\right|
\geq  cn_l^\kappa\mu^{n_l}
\end{equation}
and thus, by~\eqref{6e2},
\begin{equation} \label{6g3}
|\alpha_{n_l}|\geq c |\rho|^{N {n_l}^2} n_l^\kappa\mu^{n_l} 
= c |q|^{n_l^2/(2\sigma_k)} n_l^\kappa\mu^{n_l} ,
\end{equation}
if $l$ is sufficiently large.

For $n\in\N$ and $r>0$ we have $M(r,f)\geq |\alpha_n|r^n$.
Choosing
\begin{equation} \label{6k}
r=r_l:=\frac{1}{\mu}\exp\!\left(\frac{-\log|q|}{\sigma_k}n_l\right)
\quad\text{so that}\quad
n_l=\frac{\sigma_k\log(\mu r_l)}{-\log|q|}
\end{equation}
it follows from~\eqref{6g3} that
\begin{equation} \label{6l}
\begin{aligned} 
\log M(r_l,f)
&\geq  \log|\alpha_{n_l}|+{n_l}\log r_l
\\ &
\geq \frac{n_l^2}{2\sigma_k}\log|q|+\kappa\log n_l+n_l\log \mu +n_l\log r_l+\log c
\\ &
= \frac{n_l^2}{2\sigma_k}\log|q|+\kappa\log n_l+n_l\log (\mu r_l) +\log c .
\end{aligned} 
\end{equation}
Inserting the value of $n_l$ given by~\eqref{6k} yields that
\begin{equation} \label{6l1}
\begin{aligned} 
\log M(r_l,f)
&
\geq \frac{\sigma_k}{-2\log |q|}\left(\log \mu r_l\right)^2
+\kappa\log \!\left( \frac{\sigma_k\log(\mu r_l)}{-\log|q|}\right) + \log c 
\\ &
= \frac{\sigma_k}{-2\log |q|}\left(\log \mu r_l\right)^2+\kappa\log\log r_l + \OO(1).
\end{aligned} 
\end{equation}
Since $\log M(r,f)$ is convex in $\log r$, the conclusion now follows from~\eqref{104b1},
\eqref{6l1} and Lemma~\ref{la11}, 
applied with
\begin{equation} \label{6m0}
A(x):=\log M\!\left(e^x,f\right)
, \quad
B(x):= \frac{\sigma_k}{-2\log |q|}(x+\log\mu)^2+\kappa\log x
\end{equation}
and $x_l:=\log r_l$.  Note that $x_{l+1}=x_l+\OO(1)$ by~\eqref{6k} since $n_{l+1}\leq n_l+s$.
\end{proof}

\section{Proof of Theorem~\ref{zeros}} \label{proof-zeros}
We first consider the case that the roots of $P_k$ of modulus $\mu$ are simple.
In this case the reasoning is similar to the one in~\cite{Bergweiler2003},
with various modifications though.
Again we will use the following lemma~\cite[Lemma~3]{Bergweiler2003}
which says that the theta function is large except in small neighborhoods of its zeros.
\begin{lemma} \label{asymptheta}
Suppose that $0<|q|<1$ and  that $z\in\C$, $|z|>1$.
Define $\nu\in \N$ by $|q|^{2-2\nu}<|z|\leq |q|^{-2\nu}$.
Then, uniformly as $z\to\infty$,
\begin{equation} \label{asymptheta1}
\log|\theta(z,q)|=\frac{(\log |z|)^2}{-4\log |q|}
+\log |1+q^{2\nu-1}z|+\OO(1).
\end{equation}
\end{lemma}
We note that the case $\kappa=0$ of Lemma~\ref{la10} follows from Lemma~\ref{asymptheta}.

\begin{proof}[Proof of Theorem~\ref{zeros} in the case of simple roots]
Since we assume that the roots of $P_k$ of modulus $\mu$ are simple,
the polynomials $Q_j$ in~\eqref{asymp-coeff} are constant. Thus
\begin{equation} \label{asymp-coeff1}
\alpha_n\rho^{-N n^2}
=\sum_{\{j\colon |\lambda_j|=\mu\}} \gamma_j\lambda_j^n +\OO\!\left(\delta^n\right)
\end{equation}
with certain constants $\gamma_j$ that do not all vanish.
With $\lambda:=\lambda_1$ and $\omega_j:=\lambda_j/\lambda$ we have
\begin{equation} \label{asymp-coeff2}
\alpha_n\rho^{-N n^2}
=\sum_{\{j\colon |\lambda_j|=\mu\}} \gamma_j\omega_j^n \lambda^n+\OO\!\left(\delta^n\right).
\end{equation}
Put 
\begin{equation} \label{defM0N0}
M_0:=LM
\quad\text{and}\quad
N_0:=LN. 
\end{equation}
By hypothesis, we have $\omega_j^{M_0}=1$ for all~$j$. Hence there exists 
$(\eta_0,\dots,\eta_{M_0-1})\in\C^{M_0}\setminus \{(0,\dots,0)\}$ such that if 
$n\equiv r \pmod {M_0}$, then 
\begin{equation} \label{5a}
\sum_{\{j\colon |\lambda_j|=\mu\}} \gamma_j\omega_j^n = \eta_r .
\end{equation}
As in~\cite[Section~4]{Bergweiler2003} we shall compare $f$ with
\begin{equation}
\label{defF}
F(z):= \sum_{n=-\infty}^\infty \eta_{r_n}\rho^{N n^2}\lambda^n  z^n,
\end{equation}
where $r_n\in\{0,1,\dots,M_0-1\}$ is chosen such that $n\equiv r_n \pmod{M_0}$.
We shall see that the zeros of $F$ lie on certain geometric progressions.
The idea is to prove that the zeros of $f$ are close to those of~$F$.

Note that, by~\eqref{asymp-coeff2}, the Taylor coefficients $u_n$ of the difference
\begin{equation}
\label{defR}
R(z):=f(z)-F(z)= \sum_{n=-\infty}^\infty u_n  z^n
\end{equation}
satisfy 
\begin{equation}
\label{u_n=O}
u_n=\OO\!\left( |\rho|^{N n^2} \delta^n\right)
\end{equation}
as $n\to \infty$, and thus are small compared to those of~$F$.

The advantage of considering the function $F$ instead of $f$ is that it satisfies the 
simple functional equation
\begin{equation} \label{equationF}
F(z)=Az^{M_0} F\!\left(q^{N}z\right)
\end{equation}
with 
\begin{equation} \label{defA}
A:=\lambda^{M_0} \rho^{NM_0^2}.
\end{equation}
This implies that $F$ can be expressed as product of theta functions.
In fact, it was shown in~\cite[Theorem~4]{Bergweiler2003} that if $F$ satisfies~\eqref{equationF},
$F(z)\not\equiv 0$, and $p\in\C$ is chosen with 
\begin{equation} \label{defp}
p^2=q^{N_0}, 
\end{equation}
then there exist $C$ in $\C\setminus\{0\}$ and $z_1,z_2,\dots,z_{M_0}\in\C\setminus\{0\}$ satisfying
\begin{equation} \label{A=prod}
\prod_{l=1}^{M_0}z_l=\frac{(-1)^{M_0}}{A}
\end{equation}
such that
\begin{equation} \label{F=prod}
F(z)=C\prod_{l=1}^{M_0} \theta\!\left(-\frac{z}{pz_l},p\right).
\end{equation}
Jacobi's triple product identity~\eqref{jacobi}
implies that the zeros of $F$ are given by $M_0$ geometric progressions.
As mentioned, we will prove that the zeros of $f$ are close to these geometric progressions
since $f$ and $F$ are close.

In order to do so we proceed as in~\cite{Bergweiler2003} and note first that
\begin{equation} \label{9a}
\begin{aligned}
\log |R(z)| 
&\leq \log M\!\left(\delta |z|,\theta(\cdot,|\rho|^N)\right) +\OO(1)
\\ &
\leq \frac{(\log\delta |z|)^2}{-4\log\!\left(|\rho|^N\right)}+\OO(1)
= \frac{\sigma_k (\log\delta |z|)^2}{-2\log|q|}+\OO(1)
\\ &
= \frac{\sigma_k}{-2\log|q|}(\log|z|)^2+\frac{\sigma_k\log\delta}{-\log|q|}\log|z|+\OO(1)
\end{aligned}
\end{equation}
by~\eqref{u_n=O}, Lemma~\ref{la10} and~\eqref{6e2}.
For $1\leq l\leq M_0$ we now choose the integer $\nu_l$ such that
\begin{equation} \label{9a1}
|p|^{3-2\nu_l}\leq \left|\frac{z}{z_l}\right|< |p|^{1-2\nu_l}
\end{equation}
and thus $|p|^{2-2\nu_l}\leq |z/pz_l|< |p|^{-2\nu_l}$. 
Putting $n_l:=\nu_l-1$ we deduce, using~\eqref{F=prod} and Lemma~\ref{asymptheta}, that
\begin{equation} \label{9b}
\begin{aligned}
\log |F(z)| 
&= 
 \sum_{l=1}^{M_0} \log\left|\theta\!\left( -\frac{z}{pz_l},p\right)\right|+\log|C|
\\ &
= \sum_{l=1}^{M_0} \frac{(\log|z/pz_l|)^2}{-4\log|p|}
+ \sum_{l=1}^{M_0} \log\left|1-\frac{p^{2n_l}z}{z_l}\right| +\OO(1).
\end{aligned}
\end{equation}
Now, by~\eqref{defrho}, \eqref{defA}, \eqref{defp} and~\eqref{A=prod},
\begin{equation} \label{9c}
\begin{aligned}
\sum_{l=1}^{M_0}\log|pz_l|
&=
M_0\log|p|+
\sum_{l=1}^{M_0}\log|z_l|
=M_0\log|p|-\log |A|
\\ &
=M_0\log|p|-M_0\log \mu-NM_0^2\log|\rho|
=-M_0\log \mu
\end{aligned}
\end{equation}
and thus
\begin{equation} \label{9d}
\begin{aligned}
\sum_{l=1}^{M_0}\left(\log\!\left|\frac{z}{pz_l}\right|\right)^2
&=
\sum_{l=1}^{M_0}\left(\log|z|-\log|pz_l|\right)^2
\\ &
=M_0(\log|z|)^2-
2\left( \sum_{l=1}^{M_0}\log|pz_l|\right)\log|z| +\OO(1)
\\ &
=M_0(\log|z|)^2+(2 M_0\log\mu)\log|z| +\OO(1).
\end{aligned}
\end{equation}
Since 
\begin{equation} \label{9d1}
\frac{M_0}{4\log|p|}=\frac{M_0}{2N_0\log|q|}=\frac{M}{2N\log|q|}=\frac{\sigma_k}{2\log|q|} 
\end{equation}
by~\eqref{defp}, \eqref{defM0N0} and~\eqref{MN}
we deduce from~\eqref{9d}  that~\eqref{9b} takes the form
\begin{equation} \label{9e}
\begin{aligned}
\log |F(z)| 
&= 
\frac{\sigma_k}{-2\log|q|} (\log|z|)^2 + \frac{\sigma_k\log\mu}{-\log|q|} \log|z|
\\ & \qquad
+ \sum_{l=1}^{M_0} \log \left|1-\frac{p^{2n_l}z}{z_l}\right|
+\OO(1).
\end{aligned}
\end{equation}
Combining~\eqref{9a} and~\eqref{9e} yields that
\begin{equation} \label{log|R}
\log |R(z)|
\leq \log |F(z)| - \sum_{l=1}^{M_0} \log \left|1-\frac{p^{2n_l}z}{z_l}\right|
-C_1\log |z| + C_2,
\end{equation}
with
\begin{equation} \label{C_1=}
C_1 :=\frac{\sigma_k\log(\delta/\mu)}{\log|q|}
=\frac{M\log(\delta/\mu)}{N\log|q|}
\end{equation}
and a further constant $C_2$.
This corresponds to~\cite[(4.18)]{Bergweiler2003}.
However, in the situation of~\cite{Bergweiler2003} we had $\delta=\rho^2\mu$ and thus $C_1=1/N$.

If $\zeta$ is a zero of $f$, then $|F(\zeta)|=|R(\zeta)|$.
As in~\cite{Bergweiler2003} it then follows from from~\eqref{log|R} and~\eqref{C_1=} that 
\begin{equation} \label{extra1}
\sum_{l=1}^{M_0}
\log \left|1-\frac{p^{2n_l}\zeta}{z_l}\right|
\leq -C_1\log |\zeta| +C_2.
\end{equation}
Hence there exists $l\in \{1,2,\dots,M_0\}$ with
\begin{equation} \label{extra2}
\log \left|1-\frac{p^{2n_l}\zeta}{z_l}\right|
\leq -\frac{C_1}{M_0}\log |\zeta| +\frac{C_2}{M_0}.
\end{equation}
This implies that
\begin{equation} \label{5b}
 \left|1-\frac{p^{2n_l}\zeta}{z_l}\right|=\OO\!\left(|\zeta|^{-C_1/M_0}\right)
\end{equation}
so that $\zeta\sim z_lp^{-2n_l}=z_lq^{-N_0n_l}$ as $\zeta\to\infty$. In fact,
\begin{equation} \label{5c}
\zeta=z_lq^{-N_0n_l}\left(1+\OO\!\left(|q|^{N_0 n_l C_1/M_0} \right)\right).
\end{equation}
Using~\eqref{C_1=} and~\eqref{defM0N0} we see that 
\begin{equation} \label{5d}
|q|^{N_0  C_1/M_0}=\exp\!\left(\frac{N_0C_1}{M_0}\log|q|\right)
=\exp\!\left(\frac{N_0 M}{M_0 N}\log\frac{\delta}{\mu}\right)
=\exp\!\left(\log\frac{\delta}{\mu}\right)
=\frac{\delta}{\mu}.
\end{equation}
Thus~\eqref{5c} may also be written in the form 
\begin{equation} \label{5e}
\zeta=z_lq^{-N_0n_l}\left(1+\OO\!\left(\!\left(\frac{\delta}{\mu} \right)^{\!n_l}\right)\!\right).
\end{equation}
Recalling that $z_l p^{-2n_l}= z_lq^{-N_0n_l}$ is a zero of $F$
we thus see that every zero of $f$ of large modulus is close to a zero of~$F$.

In turn, we will see that near every zero of $F$ of large modulus there is indeed a zero of~$f$.
Let $m_l$ denote the cardinality of the set of all $j\in\{1,\dots,M_0\}$ for which 
$z_jp^{-2n_j}=z_lp^{-2n_l}$. Then $F$ has a zero of multiplicity $m_l$ at $z_l p^{-2n_l}$.
As in~\cite{Bergweiler2003} we will use Rouch\'e's theorem to show that 
$f$ has $m_l$ zeros near $z_l p^{-2n_l}$.

In order to do so we note 
that if $j\in\{1,\dots,M_0\}$ is such that $z_jp^{-2n_j}\neq z_lp^{-2n_l}$, then
$z_jp^{-2n_j}$ and $z_l p^{-2n_l}$ lie on different geometric progressions.
Thus there exists $\eta>0$ such that 
\begin{equation} \label{5e5}
\left|\frac{z_jp^{-2n}}{z_l p^{-2n_l}}-1\right|\geq\frac{\eta}{|z_l|} 
\end{equation}
and hence
\begin{equation} \label{5e6}
\left|z_jp^{-2n}-z_l p^{-2n_l}\right|\geq\eta |p|^{-2n_l}
\end{equation}
for all such $j$ and all $n\in\Z$.
We conclude that if $0<\varepsilon<\min\{\eta,1-|p|^2\}$, then
the disk of radius $\varepsilon|p|^{-2n_l}$ around the zero $z_lp^{-2n_l}$ of $F$ 
contains no other zeros of $F$.

We want to apply Rouch\'e's theorem to this disk and thus have to show
\begin{equation} \label{5e1}
|F(z)-f(z)|<|F(z)|
\quad\text{for}\ \left|z-z_lp^{-2n_l}\right|= \varepsilon|p|^{-2n_l}.
\end{equation}
Recalling that $R(z)=f(z)-F(z)$ we deduce from~\eqref{log|R} that~\eqref{5e1} holds if 
\begin{equation} \label{5e2}
- \sum_{j=1}^{M_0} \log \left|1-\frac{p^{2n_j}z}{z_j}\right|< C_1\log |z| - C_2
\quad\text{for}\ \left|z-z_lp^{-2n_l}\right|= \varepsilon|p|^{-2n_l}.
\end{equation}

To prove~\eqref{5e2} we suppose now that
\begin{equation} \label{5e3}
\left|z-z_lp^{-2n_l}\right|= \varepsilon|p|^{-2n_l}.
\end{equation}
If $j\in\{1,\dots,M_0\}$ is such that $z_jp^{-2n_j}=z_lp^{-2n_l}$, then this implies that 
\begin{equation} \label{5e4}
 \left|1-\frac{p^{2n_j}z}{z_j}\right|
= \left|1-\frac{p^{2n_l}z}{z_l}\right|
= \frac{|p|^{2n_l}}{|z_l|} \left|z_lp^{-2n_l}-z\right|= \frac{\varepsilon}{|z_l|}.
\end{equation}
If $j\in\{1,\dots,M_0\}$ is such that $z_jp^{-2n_j}\neq z_lp^{-2n_l}$, then~\eqref{5e6}
and~\eqref{5e3} yield that
\begin{equation} \label{5e7}
\begin{aligned} 
\left|1-\frac{p^{2n_j}z}{z_j}\right|
&=\frac{|p|^{2n_j}}{|z_j|}\left| z_jp^{-2n_j}-z\right|
\\ &
\geq \frac{|p|^{2n_j}}{|z_j|} 
\left( \left| z_jp^{-2n_j}- z_l p^{-2n_l} \right| -\left|z-z_lp^{-2n_l}\right| \right)
\\ &
\geq \frac{|p|^{2n_j}}{|z_j|} (\eta-\varepsilon) |p|^{-2n_l}
= \frac{|p|^{2(n_j-n_l)}}{|z_j|} (\eta-\varepsilon).
\end{aligned} 
\end{equation}
It follows from~\eqref{9a1} that 
\begin{equation} \label{5e8}
|p|^{2(n_j-n_l)} \frac{|z_l|}{|z_j|}
=\frac{|p|^{2n_j}|z_l|}{|p|^{2n_l} |z_j|}
=\frac{|p|^{2\nu_j}|z_l|}{|p|^{2\nu_l} |z_j|}
\geq \frac{|p^{3}z|}{|pz|}
=|p|^2.
\end{equation}
Together with the previous inequality we thus find that 
if $z_jp^{-2n_j}\neq z_lp^{-2n_l}$ and $z$ satisfies~\eqref{5e3},
then
\begin{equation} \label{5e9}
\left|1-\frac{p^{2n_j}z}{z_j}\right|
\geq \frac{|p|^2}{|z_l|}(\eta-\varepsilon).
\end{equation}
Combining~\eqref{5e4} and~\eqref{5e9} we deduce that if $z$ satisfies~\eqref{5e3},
then
\begin{equation} \label{5e10}
- \sum_{j=1}^{M_0} \log\left|1-\frac{p^{2n_j}z}{z_j}\right|< 
- m_l \log\!\left(\frac{\varepsilon}{|z_l|}\right)
- (M_0-m_l) \log\!\left(\frac{|p|^2}{|z_l|}(\eta-\varepsilon)\right).
\end{equation}
It follows that~\eqref{5e2} and hence~\eqref{5e1} are satisfied if  $n_l$ and
hence $|z|$ are sufficiently large.
As explained above, Rouch\'e's theorem now yields that $f$ has $m_l$ zeros in the disk
of radius $\varepsilon |p|^{-2n_l}$ around the point $z_lp^{-2n_l}=z_l q^{-N_0n_l}$.

Moreover, the argument used to obtain~\eqref{5c} shows that all zeros of sufficiently
large modulus are contained in such a disk.
It follows that there are $m_l$ zeros $\zeta$ with the asymptotics~\eqref{5e},
and that all zeros of $f$ are covered by these asymptotics for some $l$.
This completes the proof of Theorem~\ref{zeros} in the case that the roots of $P_k$ of
modulus $\mu$ are all simple.
\end{proof}

\begin{remark} \label{rem1}
It was shown in~\cite{Bergweiler2003} that with $m_l$ as in the above proof, the error term in~\eqref{107b}
may be improved to
\begin{equation}\label{107c}
z_{n,\nu}= A_\nu q^{-N n}\left(1+\OO\!\left(|q|^{n/m_\nu}\right)\right).
\end{equation}
Similarly we could improve the error term in~\eqref{107a}.
\end{remark}

In order to consider the case where $P_k$ has a multiple root of modulus $\mu$,
we first prove some auxiliary results.
Let $\Delta:=\{z\in\C\colon|z|>1\}$.
\begin{lemma} \label{lemma-ff'}
Let $(a_n)$ be a sequence in $\Delta$ such that $a_n\to\infty$ as $n\to\infty$ and
$|a_n|\leq|a_{n+1}|$ for all~$n\in\N$.

Suppose that there exists $K>1$ and $N\in\N$ such that 
$|a_{n+N}|\geq K |a_n|$ for all~$n\in\N$; that is,
each annulus $\{z\colon r<|z|\leq Kr\}$ contains at most $N$ of the points~$a_n$.
Then the infinite product 
\begin{equation}\label{7a-1}
\prod_{n=1}^\infty \left( 1-\frac{z}{a_n}\right)
\end{equation}
converges locally uniformly in~$\C$.

Let $g\colon\Delta\to\C\setminus\{0\}$ be a holomorphic function  that extends 
meromorphically to $\Delta\cup\{\infty\}$ and 
let $G\colon\Delta\to\C$ be defined by
\begin{equation}\label{7a0}
G(z):=g(z) \prod_{n=1}^\infty \left( 1-\frac{z}{a_n}\right).
\end{equation}
Then for each $m\in\N$ there exists $R>1$ such that the zeros of $G^{(m)}$ 
in $\{z\colon |z|>R\}$ are given by a sequence $(b_n)_{n\geq n_0}$ satisfying
$b_n\sim a_{n}$ as $n\to\infty$.
More precisely, we have
\begin{equation}\label{7a1}
b_n=\left(1+\OO\!\left(\frac{1}{n}\right)\!\right) a_{n}.
\end{equation}
\end{lemma}

\begin{proof}[Proof of Lemma~\ref{lemma-ff'}]
Since $|a_{n+1}|\geq |a_n|$ and $|a_{n+N}|\geq K |a_n|$ for all~$n\in\N$ 
we find by induction that if $p\in\N$ and $q\in\{1,\dots,N\}$, then
\begin{equation}\label{7a2}
|a_{pN+q}|\geq |a_{pN+1}|\geq K^p |a_1|\geq K^p .
\end{equation}
It follows that
\begin{equation}\label{7a3}
\sum_{n=1}^\infty \frac{1}{|a_n|} 
= \sum_{q=1}^N\sum_{p=0}^\infty \frac{1}{|a_{pN+q}|} 
\leq \sum_{q=1}^N\sum_{p=0}^\infty \frac{1}{K^p} 
=\frac{NK}{K-1}<\infty.
\end{equation}
This implies that the infinite product~\eqref{7a-1} converges locally uniformly
in~$\C$.

Let $n(r)$ denote the number of points $a_j$ in the disk $\{z\colon |z|\leq r\}$.
By hypothesis  we have
\begin{equation}\label{7a4}
n(Kr)-n(r)\leq N
\end{equation}
for all $r>1$.

As $g$ extends meromorphically to $\Delta\cup\{\infty\}$,
there exist $c\in\C\setminus\{0\}$ and $M\in\Z$ such that $g(z)\sim cz^M$ as $z\to\infty$.
This implies that 
\begin{equation}\label{7a5}
 \frac{g'(z)}{g(z)} =\frac{M}{z}+\OO\!\left(\frac{1}{|z|^2}\right)
\end{equation}
as $z\to\infty$. Thus
\begin{equation}\label{7b}
\frac{zG'(z)}{G(z)}=M+\sum_{j=1}^\infty \frac{z}{z-a_j} +\OO\!\left(\frac{1}{|z|}\right).
\end{equation}
For $n\in\Z$ and $r>0$ we put
\begin{equation}\label{7b1}
X_n:=\left\{a_j\colon K^{4n-2}r<|a_j|\leq K^{4n+2}r\right\}.
\end{equation}
Note that the cardinality of $X_n$ is at most $4N$.
For $r/K<|z|\leq Kr$ we have
\begin{equation}\label{7c}
\begin{aligned}
\left|\sum_{|a_j|> K^2r}\frac{z}{z-a_j}\right|
&\leq
\sum_{|a_j|> K^2r} \frac{1}{|a_j/z|-1}
\\ &
= \sum_{n=1}^\infty \sum_{a_j\in X_n}\frac{1}{|a_j/z|-1}
\leq
\sum_{n=1}^\infty \frac{4N}{K^{4n-3}-1}<\infty
\end{aligned}
\end{equation}
and
\begin{equation}\label{7d}
\begin{aligned}
\left|\sum_{|a_j|\leq K^{-2}r}\frac{z}{z-a_j}-n(K^{-2}r)\right|
&=
\left|\sum_{|a_j|\leq K^{-2}r}\frac{a_j}{z-a_j}\right|
\leq
\sum_{|a_j|\leq K^{-2}r} \frac{1}{|z/a_j|-1}
\\ &
= \sum_{n=1}^\infty \sum_{a_j\in X_{-n}} \frac{1}{|z/a_j|-1}
\leq \sum_{n=1}^\infty \frac{4N}{K^{4n-3}-1}<\infty .
\end{aligned}
\end{equation}
It follows from the last three equations and~\eqref{7a4} that for  $r/K<|z|\leq Kr$ we have
\begin{equation}\label{7e}
\frac{zG'(z)}{G(z)}-\sum_{K^{-2}r<|a_j|\leq K^2r}\frac{z}{z-a_j} =M+n(K^{-2}r)+\OO(1)=n(r)+\OO(1)
\end{equation}
as $r\to\infty$. 

Put $C:=16NK$ and let $z$ be such that $r/K<|z|\leq Kr$ and
\begin{equation}\label{7f}
|z-a_j|\geq  \frac{Cr}{n(r)}
\quad \text{for all}\  j \ \text{with} \ 
\frac{r}{K^2}<|a_j|\leq K^2r.
\end{equation}
Since there are at most $4N$ points $a_j$ satisfying $K^{-2}r<|a_j|\leq K^2r$,
it follows that
\begin{equation}\label{7f1}
\begin{aligned}
\left|\sum_{K^{-2}r<|a_j|\leq K^2r}\frac{z}{z-a_j}\right|
&\leq
\sum_{K^{-2}r<|a_j|\leq K^2r}\frac{|z|}{|z-a_j|}
\\ &
\leq
\sum_{K^{-2}r<|a_j|\leq K^2r}\frac{Krn(r)}{Cr}\leq \frac{4NK}{C}n(r)=\frac14 n(r).
\end{aligned}
\end{equation}
For sufficiently large $r$ we  thus deduce from~\eqref{7e} that 
if $z$ satisfies $r/K<|z|\leq Kr$ and~\eqref{7f}, then
$zG'(z)/G(z)$ may 
be written in the form 
\begin{equation}\label{7g}
\frac{zG'(z)}{G(z)} = n(r)+S(z)
\quad\text{with}\ |S(z)|\leq \frac12 n(r).
\end{equation}

Suppose now that $\zeta$ is a zero of $G'$ of large modulus.
Choosing $r=|\zeta|$ it follows from~\eqref{7f} and~\eqref{7g} that there exists
$j$ such that $|\zeta-a_j|< Cr/n(r)$. Thus every zero of $G'$ is close to some zero of $G$.
On the other hand, let $a_k$ be a zero of $G$ of large modulus and put $r=|a_k|$. Let
\begin{equation}\label{7g1}
W=\bigcup_{K^{-2}r<|a_j|\leq K^2r} \left\{z\colon |z-a_j|<  \frac{Cr}{n(r)}\right\}
\end{equation}
and let $U$ be the component of $W$ that contains $a_k$.
Since $W$ is the union of at most $4N$ disks,
we see that the diameter of $U$ is at most $8NCr/n(r)$.
For sufficiently large $r$ we deduce that $U\subset \{z\colon r/K<|z|<Kr\}$.
Thus~\eqref{7g} holds for $z\in\partial U$.
Rouch\'e's theorem now yields that $G$ and $G'$ have the same number of zeros in~$U$.
Thus near every zero of $G$ there is also a zero of $G'$.
And the above estimate of the diameter of $U$ yields that for large $R$ 
we can write the zero sequence of $h'$ in the form $(c_n)_{n\geq n_0}$ with $c_n\sim a_n$
and in fact $|c_{n}-a_n|=\OO(|a_n|/n)$.

This proves the result for $m=1$. We note that we do not necessarily have
$|c_n|\leq |c_{n+1}|$. However, we have $|c_n|\leq (1+o(1))|c_{n+1}|$. Noting that we
do not really need that $|a_n|\leq |a_{n+1}|$ in the above proof for the
case $m=1$, but only that $|a_n|\leq (1+o(1))|a_{n+1}|$, the general case now follows by induction.
\end{proof}
The following lemma generalizes Lemma~\ref{asymptheta}.
\begin{lemma} \label{lemma-fM}
Let $G$ be defined by~\eqref{7a0} as in Lemma~{\rm\ref{lemma-ff'}}, with
$g$, $(a_n)$, $K$ and $N$ as there.
Then there exists $\delta>0$ and $R>0$ such that if 
$r\geq R$ and $r/K<|z|\leq Kr$, then
\begin{equation} \label{7k}
\frac{|G(z)|}{M(|z|,G)}
\geq \delta \prod_{K^{-2}r<|a_j|\leq K^2r}\left| 1-\frac{z}{a_j} \right| .
\end{equation}
\end{lemma} 
\begin{proof} 
As in the proof of Lemma~\ref{lemma-ff'} we have $g(z)\sim c z^M$ as $z\to\infty$.
We may choose $R$ such that 
$|z_1^{-M}g(z_1)|\leq 2|z_2^{-M}g(z_2)|$ whenever $|z_1|>R/K$ and
$|z_2|>R/K$. Defining $X_n$ by~\eqref{7b1} as in the proof of Lemma~\ref{lemma-ff'} we then have
\begin{equation}\label{7l}
\frac{|G(z)|}{M(|z|,G)}
\geq \frac12 \prod_{j=1}^\infty \frac{\left| 1-z/a_j\right|}{1+|z/a_j|}
=\frac12 \prod_{n=-\infty}^\infty \prod_{a_j\in X_n} \frac{\left| 1-z/a_j\right|}{1+|z/a_j|}.
\end{equation}
Since $X_n$ contains at most $4N$ points we deduce for $n\geq 1$ that
\begin{equation}\label{7m}
\prod_{a_j\in X_n} \frac{\left| 1-z/a_j\right|}{1+|z/a_j|}
\geq \prod_{a_j\in X_n} \frac{1-Kr/|a_j|}{1+Kr/|a_j|}
\geq \left( \frac{1-K^{3-4n}}{1+K^{3-4n}}\right)^{4N}
\end{equation}
and, using also that $(x-1)/(1+x)=(1-1/x)/(1+1/x)$,
\begin{equation}\label{7n}
\begin{aligned}
\prod_{a_j\in X_{-n}} \frac{\left| 1-z/a_j\right|}{1+|z/a_j|}
&\geq 
\prod_{a_j\in X_{-n}} \frac{|z/a_j| -1}{1+|z/a_j|}
= \prod_{a_j\in X_{-n}} \frac{1-|a_j/z|}{1+|a_j/z|}
\geq \left( \frac{1-K^{3-4n}}{1+K^{3-4n}}\right)^{\! 4N} .
\end{aligned}
\end{equation}
We also have 
\begin{equation}\label{7o}
\prod_{K^{-2}r<|a_j|\leq K^2r}
\frac{1}{1+|z/a_j|}\geq \left(\frac{1}{1+K^3}\right)^{\! 4N}.
\end{equation}
Combining the last four estimates and taking
\begin{equation}\label{7p}
\delta:= \frac12  \left(\frac{1}{1+K^3}\right)^{4N} 
\prod_{n=1}^\infty \left( \frac{1-K^{3-4n}}{1+K^{3-4n}}\right)^{8N}.
\end{equation}
we obtain the conclusion.
\end{proof} 

\begin{proof}[Proof of Theorem~\ref{zeros} in the case of multiple roots]
Let $\kappa$ the maximal degree of the polynomials $Q_j$ in~\eqref{asymp-coeff}.
Instead of~\eqref{asymp-coeff1} and~\eqref{asymp-coeff2} we now obtain
\begin{equation} \label{asymp-coeff1a}
\begin{aligned} 
\alpha_n\rho^{-N n^2}
&=\sum_{\{j\colon |\lambda_j|=\mu\}} \gamma_j n^\kappa \lambda_j^n +\OO\!\left(n^{k-1}\mu^n\right)
\\ &
=\sum_{\{j\colon |\lambda_j|=\mu\}} \gamma_j n^\kappa \omega_j^n \lambda^n+\OO\!\left(n^{k-1}\mu^n\right),
\end{aligned} 
\end{equation}
where $\gamma_j$ is the leading coefficient of $Q_j$ if $\deg Q_j=\kappa$ and $\gamma_j=0$ otherwise,
and $\lambda=\lambda_1$ and $\omega_j=\lambda_j/\lambda$ as before.
With $M_0=LM$ and $N_0=LN$ we again find that there exists
$(\eta_0,\dots,\eta_{M_0-1})\in\C^{M_0}\setminus \{(0,\dots,0)\}$ such that~\eqref{5a} holds if
$n\equiv r \pmod {M_0}$.
The idea is now to compare $f$ not with the function $F$ given by~\eqref{defF} and~\eqref{F=prod},
but with $G:=F^{(\kappa)}$.  

Instead of~\eqref{defR} and~\eqref{u_n=O} we now have
\begin{equation} \label{defRa}
R(z):=f(z)-G(z)=f(z)-F^{(\kappa)}(z)= \sum_{n=-\infty}^\infty u_n  z^n
\end{equation}
where
\begin{equation} \label{u_n=Oa}
u_n=\OO\!\left( |\rho|^{N n^2} n^{\kappa-1}\mu^n\right)
\end{equation}
as $n\to \infty$.
Thus $\log M(r,R)\leq \log M(\mu r,\theta_{\kappa-1}(\cdot,|\rho|^N))+\OO(1)$.
Lemma~\ref{la10} yields together with~\eqref{6e2} that
\begin{equation} \label{8ba}
\begin{aligned}
\log M(r, R)
&\leq \frac{(\log \mu r)^2}{-4\log\!\left(|\rho|^N\right)}+(\kappa-1)\log\log r+\OO(1)
\\ &
=\frac{\sigma_k}{-2\log|q|} (\log \mu r)^2+(\kappa-1)\log\log r+\OO(1).
\end{aligned}
\end{equation}
On the other hand, the growth of $f$ is given by Theorem~\ref{growth}.
We conclude that 
\begin{equation} \label{7q}
M(r,R)=\OO\!\left(\frac{M(r,f)}{\log r}\right)
\end{equation}
and hence
\begin{equation} \label{7r}
M(r,G)\sim M(r,f)
\end{equation}
as $r\to\infty$.

As noted earlier, \eqref{F=prod} and Jacobi's  triple product identity~\eqref{jacobi} yield that
the zeros of $F$ form $M_0$ geometric progressions.
Lemma~\ref{lemma-ff'} implies that the zeros of $G=F^{(\kappa)}$ in $\{z\colon |z|>1\}$
are asymptotic to these geometric progressions. We may write these zeros as a sequence $(a_n)$ 
satisfying the hypotheses of Lemmas~\ref{lemma-ff'}
and~\ref{lemma-fM} so that $G$ has the form~\eqref{7a0}.

By the definition of $R$ we have
\begin{equation} \label{7r1}
\frac{f(z)}{G(z)}=1+ \frac{R(z)}{G(z)}.
\end{equation}
Lemma~\ref{lemma-fM}, together with~\eqref{7q} and~\eqref{7r},
yields that there exists a constant $H>1$ such that if $r$ is sufficiently large and $r/K<|z|\leq Kr$, then
\begin{equation} \label{7r2}
\begin{aligned}
\left| \frac{R(z)}{G(z)}\right|
& 
\leq \frac{M(|z|,R) }{\delta M(|z|,G)} 
\prod_{K^{-2}r<|a_j|\leq K^2r}\left| 1-\frac{z}{a_j} \right|^{-1}
\\ &
\leq
\frac{H}{\log|z|}\prod_{K^{-2}r<|a_j|\leq K^2r}\left| \frac{a_j}{z-a_j} \right|.
\end{aligned}
\end{equation}
Put $C:=2HK^2$ and $\alpha:=1/(4N)$. Let $z$ be such that $r/K<|z|\leq Kr$ and
\begin{equation}\label{7r3}
|z-a_j|\geq  \frac{Cr}{(\log r)^\alpha}
\quad \text{for all}\  j \ \text{with} \ 
\frac{r}{K^2}<|a_j|\leq K^2r.
\end{equation}
For large $r$ we then have
\begin{equation}\label{7f2}
\begin{aligned}
\prod_{K^{-2}r<|a_j|\leq K^2r}\left| \frac{a_j}{z-a_j} \right|
&\leq \prod_{K^{-2}r<|a_j|\leq K^2r} \frac{K^2r(\log r)^\alpha}{Cr}
\leq\left( \frac{(\log r)^\alpha}{2H}\right)^{4N}
\\ &
=\frac{\log r}{(2H)^{4N}}
\leq \frac{\log(K|z|)}{(2H)^{4N}}
\leq \frac{\log|z|}{2H}.
\end{aligned}
\end{equation}
Together with~\eqref{7r2} this yields that
\begin{equation} \label{7s}
\left| \frac{R(z)}{G(z)}\right|
\leq \frac12 
\end{equation}
if $z$ satisfies $r/K<|z|\leq Kr$ and~\eqref{7r3}, provided $r$ is sufficiently large.

We can now deduce from~\eqref{7r1} and Rouch\'e's theorem that 
if  $U$ is a component of the union of the disks $\{z\colon |z-a_j|< C r/(\log r)^\alpha\}$ which is
contained in the annulus $\{z\colon r/K<|z|<K\}$, then $f$ and $G$ have the same number of zeros in $U$.
Similarly as in the proof of Lemma~\ref{lemma-ff'} we see that the diameter
of $U$ is at most $8NCr/(\log r)^\alpha$. Moreover, all zeros of $f$ are in such components.
This implies that the zeros of $f$ are asymptotic to those of $G$, and thus asymptotic to
finitely many geometric progressions.
\end{proof}
\section{Equations of order two}\label{order2}
We discuss equation~\eqref{101}
in the special case that $m=2$,  $\deg a_j=j$ for $j\in\{0,2\}$ and $\deg a_1\leq 1$.
We shall see that in this case the hypothesis
on the roots of the characteristic polynomial that was made in Theorem~\ref{zeros} is sharp.

It is no loss of generality to assume that $a_0(z)\equiv 1$. We may also assume
that the leading coefficient of $a_2$ is equal to $q^2$,
since this can be achieved by replacing $f(z)$ by $f(cz)$ for a suitable constant~$c$.
(We normalize this coefficient to $q^2$ and not to~$1$ because this simplifies some of the formulas
below and, more importantly, agrees with the notation in~\cite{Bergweiler2003}.)
We are thus considering the equation
\begin{equation}\label{101a}
f(z)+(a_{1,1}z+a_{1,0})f(qz)+(q^2 z^2+a_{2,1}z+ a_{2,0})f(q^2z)=b(z)
\end{equation}
with a polynomial~$b$.

For the equation~\eqref{101a} we have $K=1$, $j_0=0$, $j_1=2$, $M=N=2$ and $\sigma_1=1$.
The corresponding segment of $\partial P$ has $(0,0)$ and~$(2,2)$ as its endpoints.
If $a_{1,1}\neq 0$, it also contains the point $(1,d(1))=(1,1)$.
The Newton-Puiseux diagram corresponding to this equation
is shown in Figure~\ref{np-diagram3}. The points $(j,d(j))$ are marked,
assuming $a_{1,1}\neq 0$ so that $d(1)=1$.
\begin{figure}[!htb]
\centering
\begin{tikzpicture}[scale=1.2,>=latex](-0.1,-0.1)(5.1,3.1)
\filldraw[gray!10] (0,0) -- (2,2) -- (3,2) -- (3,-0.7) -- (0,-0.7);
\draw[thick,-] (0,0) -- (2,2) ;
\draw[->] (-0.3,0) -- (3.0,0);
\draw[->] (0,-0.7) -- (0,2.3);
\foreach \x in {1,...,2}
   {
    \draw (\x,-2pt) -- (\x,2pt);
    \node[below] at (\x,-2pt) {$\x$};
   }
\foreach \y in {1,...,2}
   {
    \draw (-2pt,\y) -- (2pt,\y);
    \node[left] at (-2pt,\y) {$\y$};
   }
\filldraw[black] (0,0) circle (0.05);
\filldraw[black] (1,1) circle (0.05);
\filldraw[black] (2,2) circle (0.05);
\end{tikzpicture}
\caption{The Newton-Puiseux diagram of equation~\eqref{101a}.}
\label{np-diagram3}
\end{figure}
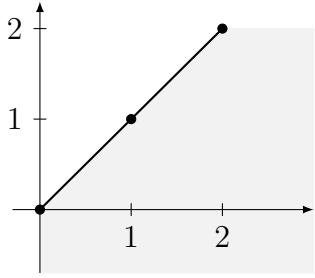

In order to be consistent with the terminology of~\cite[Example~1 and Theorem~3]{Bergweiler2003}
we write $a_{1,1}$ in the form 
\begin{equation} \label{e1a}
a_{1,1}=-2\rho^2\gamma.
\end{equation}
Here $\rho$ is chosen according to~\eqref{defrho}. Since $M=2$ this means that $\rho^4=q$.
We also have $a_{2,2}=q^2=\rho^8$.
For the coefficients $c_{j,i}$ defined by~\eqref{defc} we find 
that $c_{1,1}=a_{1,1}\rho^{-2}=-2\gamma$ and $c_{2,2}=a_{2,2}\rho^{-8}=1$.
The characteristic polynomial defined by~\eqref{p_k} thus takes the form
\begin{equation} \label{e1}
P_1(z)=  z^2-2\gamma z+1.
\end{equation}
Let $\lambda_{1,2}:=\gamma\pm\sqrt{\gamma^2-1}$ be the roots of $P_1$. 

We distinguish two cases, depending on whether these roots are distinct or not.

\begin{case1}
Let $\gamma\neq \pm 1$ so that $\lambda_1\neq\lambda_2$.
Theorem~\ref{growth} yields that the growth of a transcendental solution $f$ is given by
\begin{equation} \label{103c}
\log M(r,f)=  \frac{1}{-2\log |q|} \left(\log\left|\lambda_j\right|r\right)^2 +\OO(1)
\end{equation}
for some $j\in\{1,2\}$.

The hypotheses of Theorem~\ref{zeros} are satisfied if $|\lambda_1|\neq |\lambda_2|$
or if $|\lambda_1|=|\lambda_2|$ and $\lambda_2/\lambda_1$ is a root of unity.
The only case where they are not satisfied is when
$\lambda_1/\lambda_2=e^{2\pi i\eta}$ for some $\eta\in\R\setminus\Q$.
Since $\lambda_1\lambda_2=1$ this yields that $\lambda_{1,2}=e^{\pm i\eta\pi}$ and hence
$\gamma=(\lambda_1+\lambda_2)/2=\cos(\eta\pi)$.
We conclude that the zeros are asymptotic to finitely many geometric progressions if 
$a_{1,1}=-2\rho^2\gamma=\pm 2\sqrt{q}\gamma$ is not of the form
\begin{equation} \label{101b}
a_{1,1}=\pm 2\sqrt{q}\cos(\eta \pi), 
\quad\text{with}\ \eta\in\R\setminus\Q.
\end{equation}

The argument also shows that 
if  $\gamma\notin [-1,1]$ or, equivalently, if $a_{1,1}/\sqrt{q}\notin [-2,2]$, then
$|\lambda_1|\neq |\lambda_2|$. Thus $L=1$ in Theorem~\ref{zeros}. Since $M=2$ this yields
that the zeros are asymptotic to at most two geometric progressions in this case.

On the other hand, 
it was shown in~\cite[Example~1 and Theorem~3]{Bergweiler2003} that 
the zeros need not be asymptotic to finitely many geometric progressions if 
$\gamma=\pm \cos(\eta\pi)$ with $\eta\in\R\setminus\Q$ so that $a_{1,1}$ has the form~\eqref{101b}.
In fact, \cite[Example~1]{Bergweiler2003} says that
if $a_{1,0}=a_{2,0}=a_{2,1}=0$ so that, assuming $\gamma=\cos(\eta\pi)$, the equation takes the form
\begin{equation} \label{e3}
f(z)-2\gamma \rho^2 z f(qz)+ q^2 z^2 f(q^2 z) = b(z),
\end{equation}
and if $0<q<1$ and $c_1,c_2\in\C$, then for 
\begin{equation} \label{e3a}
b(z):= \left( -(c_1+c_2)\gamma \rho^2
+i(c_1-c_2)\rho^2\sqrt{1-\gamma^2}\right)z+c_1+c_2
\end{equation}
the series~\eqref{105} defines a solution $f$ of~\eqref{e3} for
\begin{equation} \label{e4}
\alpha_n=c_1 \lambda_1^n \rho^{2n^2}+c_2 \lambda_2^n \rho^{2n^2}
=c_1 e^{i\eta n} \rho^{2n^2}+c_2 e^{-i\eta n} \rho^{2n^2}.
\end{equation}
It was shown in~\cite[Theorem~3]{Bergweiler2003}  that if $c_1,c_2\neq 0$ and $|c_1|\neq |c_2|$,
then the arguments of the zeros of $f$ are dense in
some subinterval of $[-\pi,\pi]$, but not dense in $[-\pi,\pi]$.
In particular, the zeros are not asymptotic to a finite number of geometric progressions.

We conclude that for the equation~\eqref{101a} the hypothesis posed in Theorem~\ref{zeros} on the 
zeros of $P_k$ is not only sufficient but also necessary in order to conclude that the 
zeros of every entire solution $f$ are asymptotic to finitely many geometric progressions.
\end{case1}

\begin{case2}
Let $\gamma= \pm 1$ so that $\lambda_1=\lambda_2=\pm 1$.
This corresponds to the case $a_{1,1}=\mp 2\rho^2$ in~\eqref{101a}.

Recall that $\rho^4=q$. Choosing $p=\rho^2$ we have $p^2=q$.
By~\eqref{equationtheta} we have $\theta(z,p)=pz\theta(p^2z,p)$
and thus $\theta(p^4z,p)=\theta(p^2z,p)/(p^3z)$.
Differentiating these equations with respect to $z$ and 
eliminating $\theta(p^2z,p)$ from the resulting equations, we obtain
\begin{equation}\label{5f}
\theta'(z,p)-2p^3 z \theta'(p^2 z,p) +p^8 z^2 \theta'(p^4z,p),
\end{equation}
where $\theta'$ denotes the derivative with respect to $z$.
Hence $g(z):=\theta'(z/p^2,p)$ satisfies $g(z)-2p z g(p^2 z) +p^4 z^2 g(p^4z)=0$.
Since $p=\rho^2$ and $p^2=q$ this is equivalent to
\begin{equation}\label{5g}
g(z)-2\rho^2 z g(q z) +q^2 z^2 g(q^2 z)=0.
\end{equation}
This is equation~\eqref{e3} with $\gamma=1$ and $b(z)\equiv 0$.

Let $f$ be the entire function consisting of the non-negative powers in the 
Taylor series of $g$.
It follows from~\eqref{5g} and the definition of $f$ that
\begin{equation}\label{5i}
b(z):=f(z)-2\rho^2 z f(q z) +q^2 z^2 f(q^2 z)
\end{equation}
is a polynomial of degree at most~$1$.
For this polynomial $b$ the function $f$ thus satisfies~\eqref{e3} with $\gamma=1$.

Instead of defining $f$ via $\theta'$ and $g$ we could have defined $f$ also directly
via its Taylor series
\begin{equation}\label{5h}
\begin{aligned}
f(z)
&=\sum_{n=0}^\infty p^{(n+1)^2} (n+1) \left(\frac{z}{p^2}\right)^n
=\sum_{n=0}^\infty p^{n^2+1} (n+1) z^n
=p\sum_{n=0}^\infty \rho^{2n^2} (n+1) z^n.
\end{aligned}
\end{equation}
It can then be checked directly that $f$ satisfies~\eqref{e3} with $\gamma=1$
and $b(z)\equiv p$.

We show that the zeros of $f$ are asymptotic to a geometric series, but that the error
term is weaker than in~\eqref{107a}. To this end we put
\begin{equation}\label{5l}
F(z):=\sum_{n=0}^\infty p^{n^2} \left(\frac{z}{p^2}\right)^n
=\frac{1}{p} \sum_{n=0}^\infty   p^{(n-1)^2} z^{n}
\end{equation}
so that $F'(z)=f(z)/p^2$.  We note that the function $F$ has zeros $\xi_n$ satisfying 
\begin{equation}\label{5l1}
\xi_n=p^{-2n-1}\left(1+\OO\!\left(|p|^n\right)\right)
\end{equation}
as $n\to\infty$.
This follows from Lemma~\ref{asymptheta}
since $F$ consists of the non-negative powers of $\theta(z/p^2,p)$ and since by
Jacobi's triple product identity the zeros of the latter function are precisely the
points $p^{-2n-1}$ with $n\in\Z$. Alternatively, \eqref{5l1}
follows from~\cite[Theorem~2]{Bergweiler2003} or Theorem~\ref{zeros}.
Writing
\begin{equation}\label{5m}
F(z)=\prod_{j=1}^\infty \left(1-\frac{z}{\xi_j}\right)
\end{equation}
we have 
\begin{equation}\label{5n}
\frac{F'(z)}{F(z)}=\sum_{j=1}^\infty \frac{1}{z-\xi_j}.
\end{equation}
For large $|z|$ we choose $n\in\N$ such that $|p|^{-2n}\leq |z|< |p|^{-2n-2}$ and write
\begin{equation}\label{5o}
\frac{F'(z)}{F(z)}
= \frac{1}{z-\xi_n} +\sum_{j=1}^{n-1} \frac{1}{z-\xi_j} + \sum_{j=n+1}^\infty \frac{1}{z-\xi_j}
=: \frac{1}{z-\xi_n} +S_1+S_2.
\end{equation}
By~\eqref{5l1} we have  $|\xi_n|\geq |p|^{-2n-1/2}$ for large $n$ and thus
\begin{equation}\label{5q}
\begin{aligned}
|S_2|
&\leq \sum_{j=n+1}^\infty \frac{1}{|\xi_j|-|z|}
\leq \sum_{j=n+1}^\infty \frac{1}{|p|^{-2j-1/2}-|p|^{-2n-2}}
\\ &
=\sum_{k=1}^\infty \frac{|p|^{2n+2}}{|p|^{-2k+3/2}-1}
=\OO\!\left(|p|^{2n+2}\right)
=\OO\!\left(\frac{1}{|z|}\right)
\end{aligned}
\end{equation}
as $z\to\infty$.
By~\eqref{5l1} there also exists $j_0\in\N$ such that  $|\xi_j|\leq |p|^{-2j-3/2}$ for $j\geq j_0$. Thus
\begin{equation}\label{5r}
\begin{aligned}
|zS_1-(n-1)|
&=\left|\sum_{j=1}^{n-1} \frac{\xi_j}{z-\xi_j}\right|
=\left|\sum_{j=j_0}^{n-1} \frac{\xi_j}{z-\xi_j}\right|+\OO(1)
\\ &
\leq\sum_{j=j_0}^{n-1} \frac{|p|^{-2j-3/2}}{|p|^{-2n}-|p|^{-2j-3/2}}+\OO(1)
\\ &
=\sum_{k=1}^{n-1} \frac{2|p|^{2k-3/2}}{1-|p|^{2k-3/2}}+\OO(1)
=\OO(1)
\end{aligned}
\end{equation}
as $z\to\infty$. Thus
\begin{equation}\label{5s}
\frac{F'(z)}{F(z)}
= \frac{1}{z-\xi_n} + \frac{n}{z} +\OO\!\left(\frac{1}{|z|}\right)
\end{equation}
as $z\to\infty$, with $n$ defined by 
$|p|^{-2n}\leq |z|< |p|^{-2n-2}$.
Rouch\'e's theorem implies that for large $n$ the functions $F'/F$ and hence $f=F'$ 
have exactly one zero $z_n$ satisfying $|p|^{-2n}\leq |z_n|< |p|^{-2n-2}$.
This zero $z_n$ satisfies 
\begin{equation}\label{5s1}
\frac{1}{z_n-\xi_n} + \frac{n}{z_n} =\OO\!\left(\frac{1}{|z_n|}\right)
\end{equation}
 as $n\to\infty$.
Together with~\eqref{5l1} we deduce that
\begin{equation}\label{5t}
z_n =\left(1-\frac{1}{n}+o\!\left(\frac{1}{n}\right)\right) \xi_n
=\left(1-\frac{1}{n}+o\!\left(\frac{1}{n}\right)\right) p^{-2n-1}.
\end{equation}
We find that the $z_n$ are asymptotic to a geometric series, but not with the error term given
by~\eqref{107a}.

We note that the general solution of~\eqref{5g} is given by 
\begin{equation}\label{5j}
g(z)=C_1 \theta\!\left(\frac{z}{p^2},p\right)+C_2 \theta'\!\left(\frac{z}{p^2},p\right)
\end{equation}
with constants $C_1$ and $C_2$. This implies that if $c_1,c_2\in \C$, then
\begin{equation}\label{5k}
f(z)=\sum_{n=0}^\infty (c_1 n+c_2)  \rho^{2n^2} z^n
\end{equation}
solves \eqref{e3} for $\gamma=1$ and some polynomial $b$ of degree at most~$1$.
Theorem~\ref{growth} says that this function satisfies
\begin{equation} \label{103b}
\log M(r,f)=  \frac{1}{-2\log |q|} (\log r)^2 +\log\log r +\OO(1)
\end{equation}
if $c_1\neq 0$. If $c_1=0$ but $c_2\neq 0$, then it satisfies~\eqref{103c}, with $|\lambda_j|=1$.
\end{case2}

\noindent Mathematisches Seminar\\
Christian-Albrechts-Universit\"at zu Kiel\\
Ludewig-Meyn-Str.\ 4\\
24098 Kiel, Germany

\medskip
\noindent
E-mail: {\tt bergweiler@math.uni-kiel.de}

\begin{thebibliography}{99}
\bibitem{Adams1928} 
C. Raymond Adams,
On the irregular cases of the linear ordinary difference equation. 
Trans. Amer. Math. Soc. 30 (1928), no. 3, 507--541.

\bibitem{Adams1931} 
C. R. Adams, Linear $q$-difference equations.
Bull. Amer. Math. Soc. 37 (1931), no. 6, 361--400. 

\bibitem{Agarwal2007} 
Ravi P. Agarwal and Mih\'aly Pituk,
Asymptotic expansions for higher-order scalar difference equations.
Advances in Difference Equations,
 Volume 2007, Article ID 67492, 12 pages, doi:10.1155/2007/67492.

\bibitem{Andrews1976} George E.\ Andrews,
The theory of partitions. 
Encyclopedia of Mathematics and its Applications, Vol.~2.
Addison-Wesley, Reading, Massachusetts, 1976.

\bibitem{Annaby2012}
Mahmoud H. Annaby and  Zeinab S. Mansour, $q$-fractional calculus and equations.
Lect. Notes Math. 2056. Springer, Heidelberg, 2012.

\bibitem{Barnett2007}
D. C. Barnett, R. G. Halburd,  W. Morgan and  R. J. Korhonen, 
Nevanlinna theory for the $q$-difference operator and meromorphic solutions of
$q$-difference equations. 
Proc. Roy. Soc. Edinburgh Sect. A 137 (2007), no.~3, 457--474. 

\bibitem{Bergweiler2003} Walter Bergweiler and Walter K.~Hayman,
Zeros of solutions of a functional equation.
Comput. Methods Funct. Theory 3 (2003), no.~1--2, 55--78.

\bibitem{Bergweiler2002} 
Walter Bergweiler, Katsuya Ishizaki and Niro Yanagihara,
Growth of meromorphic solutions of some functional equations, I.
Aequationes Math.\ 63 (2002), no.~1--2, 140--151.

\bibitem{Bodine2009} Sigrun Bodine and D. A. Lutz,
Exponentially asymptotically constant systems of difference equations with 
an application to hyperbolic equilibria.
J.\ Diff.\ Equ.\ Appl.\ 15 (2009), no.~8--9, 821--832.

\bibitem{Cao2019}
Tingbin Cao, Huixin Dai and Jun Wang,
Nevanlinna theory for Jackson difference operators and entire solutions of $q$-difference equations.
Preprint, arxiv: 1812.10014.

\bibitem{Chiang2018}
Yik-Man Chiang and Shaoji Feng,
Nevanlinna theory of the Askey-Wilson divided difference operator. 
Adv. Math. 329 (2018), 217--272. 

\bibitem{Dreyfus2015}
Thomas Dreyfus, Building meromorphic solutions of $q$-difference equations
using a Borel-Laplace summation. 
Int. Math. Res. Not. IMRN 2015, no.~15, 6562--6587. 

\bibitem{Hilberdink2020}
Titus Hilberdink,
Asymptotics of entire functions and a problem of Hayman.
Q.~J.\ Math., https://doi.org/10.1093/qmathj/haz061.


\bibitem{Juneja1977} 
O.~P.\ Juneja, G.~P.\ Kapoor and S.~K.\ Bajpai, 
On the $(p,q)$-type and lower $(p,q)$-type of an entire function.
J.\ Reine Angew.\ Math. 290 (1977), 180--190.

\bibitem{Perron}
O. Perron, \"Uber Summengleichungen und Poincar\'esche Differenzen\-glei\-chungen.
Math. Ann. 84 (1921), no.~1--2, 1--15.

\bibitem{Poincare}
H. Poincar\'e,
Sur les \'equations lin\'eaires aux diff\'erentielles ordinaires et aux diff\'erences finies.
Amer. J. Math. 7 (1885), no.~3, 203--258.

\bibitem{Ramis1992} Jean-Pierre Ramis, 
{\rm About the growth of entire functions solutions of 
linear algebraic $q$-difference equations}.
Ann.\ Fac.\ Sci.\ Toulouse Math.\ (6) 1 (1992), no.~1, 53--94.

\bibitem{Ramis2013} Jean-Pierre Ramis, Jacques Sauloy and  Changgui Zhang,
Local analytic classification of $q$-difference equations. Ast\'erisque No. 355 (2013).
\end{thebibliography}
\end{document}